\documentclass[11pt,reqno,twoside]{article}
\usepackage{Arxiv}
\usepackage{graphicx}
\usepackage{multirow}
\usepackage{amsmath,amssymb,amsfonts}
\usepackage{amsthm}
\usepackage[title]{appendix}
\usepackage{xcolor}
\usepackage{textcomp}
\usepackage{manyfoot}
\usepackage{booktabs}
\usepackage[ruled,vlined]{algorithm2e}
\usepackage{listings}
\usepackage{bm}
\usepackage{xspace}
\usepackage{pifont}
\usepackage{subcaption}
\usepackage[top=2.35cm, bottom=2.35cm, left=3cm, right=3cm]{geometry}

\newcommand{\carp}{CARPA\xspace}
\newcommand{\carpa}{CARPA\xspace}
\newcommand{\gcarp}{gCARPA\xspace}

\newcommand{\ee}{\bm{e}}
\newcommand{\x}{\bm{x}}
\newcommand{\y}{\bm{y}}
\newcommand{\w}{\bm{w}}
\newcommand{\p}{\bm{p}}
\newcommand{\z}{\bm{z}}
\newcommand{\uu}{\bm{u}}
\newcommand{\vv}{\bm{v}}

\newcommand{\A}{\bm{A}}
\newcommand{\D}{\bm{D}}
\newcommand{\C}{\bm{C}}
\newcommand{\cF}{c_{\scriptstyle\rm{F}}}

\newcommand{\tF}{\theta_{\scriptstyle\rm{F}}}
\renewcommand{\Id}{\mathcal{I}}
\newcommand{\opF}{\mathcal{F}}

\renewcommand{\beq}{\begin{equation}}
\renewcommand{\eeq}{\end{equation}}
\renewcommand{\rank}{\mathrm{rank}}

\renewcommand{\find}{\mathrm{find}}

\renewcommand{\proj}{\mathcal{P}}
\newcommand{\iprod}[2]{\langle #1,\,#2 \rangle}
\newcommand{\norm}[1]{{|\kern-1.125pt|} #1 {|\kern-1.125pt|}}
\newcommand{\Ker}{\mathrm{ker}}

\newcommand{\tr}{\mathrm{tr}}

\renewcommand{\ran}{\mathrm{ran}}

\renewcommand{\Fix}{\mathrm{fix}}
\newcommand{\eqdef}{\stackrel{\text{\rm\tiny def}}{=}}
\newcommand{\argmin}{\mathrm{argmin}}

\newcommand{\calA}{\mathcal{A}}
\newcommand{\calB}{\mathcal{B}}
\newcommand{\calC}{\mathcal{C}}

\newcommand{\calF}{\mathcal{F}}

\newcommand{\calH}{\mathcal{H}}

\newcommand{\calM}{\mathcal{M}}
\newcommand{\calN}{\mathcal{N}}

\newcommand{\calR}{\mathcal{R}}

\newcommand{\calT}{\mathcal{T}}

\newcommand{\bbN}{\mathbb{N}}

\newcommand{\bbR}{\mathbb{R}}

\newcommand{\bmA}{\bm{A}}
\newcommand{\bmB}{\bm{B}}
\newcommand{\bmC}{\bm{C}}
\newcommand{\bmD}{\bm{D}}

\newcommand{\bmI}{\bm{I}}

\newcommand{\bmM}{\bm{M}}

\newcommand{\bmP}{\bm{P}}

\newcommand{\bmS}{\bm{S}}

\newcommand{\bma}{\bm{a}}
\newcommand{\bmb}{\bm{b}}

\newcommand{\bmf}{\bm{f}}

\newcommand{\bmz}{\bm{z}}

\theoremstyle{thmstyleone}
\newtheorem{theorem}{Theorem}[section]
\newtheorem{fact}{Fact}[section]
\newtheorem{lemma}{Lemma}[section]
\newtheorem{proposition}{Proposition}[section]
\theoremstyle{thmstyletwo}

\newtheorem{remark}{Remark}

\theoremstyle{thmstylethree}
\newtheorem{definition}{Definition}
\title{Generalized Composed Alternating Relaxed Projection Algorithm for Two-Set Feasibility Problem}
\author{
    Xinxin Li\thanks{School of Mathematics, Jilin University, P.R. China. Email: \texttt{xinxinli@jlu.edu.cn}}
    \and
    Yudong Wei\thanks{Department of Mathematics, ETH Zurich, Switzerland. Email: \texttt{yudwei@ethz.ch}. Corresponding author.}
    \and
    Hao Zhang\thanks{School of Mathematics, Jilin University, P.R. China. Email: \texttt{zhanghao1022@mails.jlu.edu.cn}}
}
\date{\today}

\begin{document}
\maketitle
\begin{abstract}
    \unboldmath We study the two-set feasibility problem of finding a point in the intersection $X\cap Y$ of closed convex sets in a Hilbert space. We propose a generalized composed alternating relaxed projection algorithm (\gcarp) that blends Douglas--Rachford-type and projection--reflection-type dynamics via an outer averaging step $\mu$ and an internal relaxation $(\gamma,\theta,\eta)$. The algorithm contains several classical projection methods as special cases. We also introduce its non-stationary variant,  in which $(\gamma_k,\theta_k,\eta_k)$ vary over iterations, and establish its convergence. For the subspace feasibility model, we derive an explicit spectral characterization via principal-angle block decompositions, yielding computable subdominant-eigenvalue factors and a minimax parameter-selection recipe in a symmetric regime that targets critical damping on principal-angle planes. Numerical experiments illustrate that the generalized relaxation and its non-stationary tuning can improve or match baseline methods in problem-dependent regimes.
\end{abstract}


\section{Introduction}\label{sec:introduction}
	In this paper, we are interested in the following feasibility problem of two sets:
	\beq\label{eq:feasi_problem}
	\find\; \x\in \calH \ \ \textrm{such that}\ \ \x\in X\cap Y,
	\eeq
	where $\calH$ is a real Hilbert space, and $X,Y\subset \calH$ are nonempty closed convex sets satisfying
	$
	X\cap Y\neq\emptyset.
	$
	Feasibility problems of the form \eqref{eq:feasi_problem} arise in a wide range of applications, including signal and image processing \cite{combettes1996convex,bekri2025robust}, compressed sensing \cite{suantai2019new,van2009algorithm}, seismic data processing \cite{gao2013convergence,smarra2020learning}, and many others. Projection-based methods are among the most popular approaches for such problems, due to their simplicity and their strong convergence theory in Hilbert spaces.
	
	A fundamental method for \eqref{eq:feasi_problem} is the method of alternating projections (MAP) \cite{gubin1967method,bauschke1993convergence}. Over the years, numerous variants have been proposed to improve practical performance or to obtain sharper convergence properties, such as relaxed alternating projection (RAP) and partial relaxed alternating projection (PRAP) \cite{bauschke2016optimal}, averaged alternating modified reflections (AAMR) \cite{aragon2019optimal}, relaxed averaged alternating reflections (RAAR) \cite{luke2004relaxed}, simultaneous projection (SP) \cite{reich2017optimal}, and generalized relaxed alternating projection (GRAP) \cite{falt2017optimal,dao2018linear}. More recently, geometric acceleration ideas have also been explored, for instance via circumcentered-reflection type schemes \cite{arefidamghani2022circumcentered,behling2021circumcentered,behling2024successive}.
	
	Projection algorithms can also be viewed through the lens of operator splitting, since the projection is a special case of the proximity operator \cite{bauschke2020correction}. Indeed, MAP can be interpreted as an instance of the backward--backward splitting method \cite{Bauschke2005}. Moreover, several projection/reflection methods recover classical splittings such as Peaceman--Rachford \cite{peaceman1955numerical} and Douglas--Rachford \cite{douglas1956numerical} under suitable parameter choices; see, e.g., \cite{bauschke2014rate,bauschke2016optimal,aragon2019optimal} and references therein.
	
	The convergence theory of projection-based fixed-point iterations has also been extensively developed. Many algorithms can be written as Krasnosel'ski\u{\i}--Mann iterations of (firmly) nonexpansive or averaged operators, enabling global convergence guarantees and rate estimates; see, e.g., \cite{liang2016convergence}. Under additional regularity such as partial smoothness, one can obtain local linear convergence \cite{Hare2007,Lewis2009,liang2017localDR}. For nonconvex feasibility, convergence properties of Douglas--Rachford type schemes have been studied as well; see, e.g., \cite{Li2016}. In the special case where $X$ and $Y$ are subspaces, sharp linear rates  can be characterized in terms of the Friedrichs angle \cite{bauschke2014rate,Bauschke2016,aragon2019optimal}; see also \cite{Dao2019,Bauschke2023,Bauschke2024} for further developments.

    \subsection{Motivation and contribution}	
	When $X$ and $Y$ are two subspaces, both MAP and Douglas--Rachford (DR) converge linearly, and the rate is governed by the Friedrichs angle between $X$ and $Y$ \cite{bauschke2014rate,bauschke2016optimal}. However, despite its appealing theory, the DR fixed-point sequence may exhibit a pronounced spiraling trajectory \cite{poon2019trajectory}, which can lead to slow practical progress on subspace feasibility problems. This geometric phenomenon motivates algorithmic designs that retain the favorable nonexpansive structure of DR while mitigating its spiral behavior.
	
	Recently, Shen et al. \cite{shen2025composed} proposed the composed alternating relaxed projection algorithm (\carp), which blends DR-type composition with a projection-reflection correction, and demonstrated significant speedup over DR on several feasibility models. In this work, we further generalize \carp by introducing relaxed reflections with parameters $\theta,\eta\in(0,1]$, leading to a generalized composed alternating relaxed projection algorithm (\gcarp), given in Algorithm \ref{alg:gcarp}. Concretely, we use
	\[
	\mathcal R_X^\theta := (1-\theta)\Id + \theta(2\proj_X-\Id),
	\qquad
	\mathcal R_Y^\eta := (1-\eta)\Id + \eta(2\proj_Y-\Id),
	\]
    where $\proj_X$ and $\proj_Y$ denote the orthogonal projections onto $X$ and $Y$, and $\Id$ is the identity operator. When $\theta=\eta=1$, the operators $\mathcal R_X^\theta$ and $\mathcal R_Y^\eta$ reduce to the standard reflections, in which case \gcarp coincides with \carp in \cite{shen2025composed}. The additional parameters $\theta$ and $\eta$ provide a flexible interpolation between projection steps and reflection steps, and allow us to unify several classical schemes within a single template. For instance, the \gcarp update can be written as a relaxed fixed-point iteration
	\[
	\z^{k+1}=(1-\mu)\z^k+\mu\,\opF^{\theta,\eta}_{\gamma}(\z^k),
	\]
	where the operator $\opF^{\theta,\eta}_{\gamma}$ is a convex combination of a DR-type averaging term and a projection term built from $\proj_Y\circ \mathcal R_X^\theta$. As special cases, one may recover relaxed DR (e.g., $\gamma=0$), and relaxed alternating projections (e.g., $\gamma=1$ and $\mu=1$), highlighting \gcarp as a unifying algorithm.
	
	The main contributions of this paper are summarized as follows.
	\begin{enumerate}
	    \item[\ding{118}] We propose the generalized composed alternating relaxed projection algorithm for the convex feasibility problem \eqref{eq:feasi_problem}. 	Compared with \carp, the additional relaxation parameters $(\theta,\eta)$ in \gcarp provide extra
		degrees of freedom to control the balance between spiraling and damping in the
		fixed-point dynamics. We establish its fixed-point characterization and global convergence under suitable parameter choices. We also introduce and analyze a non-stationary variant.
		\item[\ding{118}] For the feasibility problem of two subspaces, we derive explicit linear convergence rates of \gcarp and identify parameter regimes that yield provably optimal rates. In particular, when $\theta=\eta=1$, our results specialize to the \carp setting of \cite{shen2025composed}.
		\item[\ding{118}] We provide numerical experiments on representative feasibility models, including subspace feasibility and additional structured examples, to illustrate how the extra relaxation parameters $\theta,\eta$ can improve robustness and performance compared with existing projection/reflection methods.
	\end{enumerate}
    
	\subsection{Paper organization}
	The remainder of this paper is organized as follows. Section~\ref{sec:maths_bg} collects notation and preliminary results. In Section~\ref{sec:gcarpa}, we present \gcarp and its non-stationary extension, together with the convergence analysis. Section~\ref{sec:numerics} reports numerical experiments and comparisons with several related methods.

    \section{Mathematical background}
	\label{sec:maths_bg}
	
	Throughout, let $\calH$ be a real Hilbert space endowed with inner product $\iprod{\cdot}{\cdot}$ and the induced norm $\norm{\cdot}$.
	We write $\Id$ for the identity operator on $\calH$. 
	When we specialize to the Euclidean setting, $\bbR^{n}$ denotes the $n$-dimensional real vector space equipped with the standard inner product.
	
	We use bold uppercase (lowercase) letters to denote matrices (column vectors) and calligraphic letters to represent operators. Given a matrix $\A\in\bbR^{m\times n}$, we denote by $\Ker(\A)$, $\ran(\A)$, and $\rank(\A)$ the null space, range, and rank of $\A$, respectively. We use $\mathrm{span}(\cdot)$ to denote the linear span of a set of vectors. We denote by $\mathrm{card}(\cdot)$ the cardinality of a finite set, i.e., the number of elements in the set.
	
	\paragraph{Operators.}
	We recall several notions on nonexpansive mappings and projection-type operators; see \cite[Chapter~4]{bauschke2020correction} for further background.
	
	\begin{definition}[Nonexpansive and averaged operators]\label{def:nonexpansive-operator}
		An operator $\calF:\calH\to\calH$ is \emph{nonexpansive} if
		\[
		\norm{\calF(\x)-\calF(\x')}\le \norm{\x-\x'} \qquad \forall\,\x,\x'\in\calH.
		\]
		For $\alpha\in(0,1)$, the operator $\calF$ is called \emph{$\alpha$-averaged} if there exists a nonexpansive mapping $\calN:\calH\to\calH$ such that
		\[
		\calF = (1-\alpha)\Id + \alpha \calN .
		\]
		When $\alpha=\tfrac12$, $\calF$ is also called \emph{firmly nonexpansive}; in this case $2\calF-\Id$ is nonexpansive.
	\end{definition}

    \begin{definition}[Fixed point]
        A point $\x\in\calH$ is a \emph{fixed point} of $\calF$ if $\x=\calF(\x)$.
	   We denote the fixed-point set by
	\[
	\Fix(\calF):=\{\x\in\calH:\ \x=\calF(\x)\},
	\]
	which may be empty.
    \end{definition}
	
	\begin{definition}[Projection]\label{def:projection}
		Let $X\subset\calH$ be nonempty, closed, and convex.
		The \emph{projection} onto $X$ is
		\[
		\proj_X(\x)\eqdef \argmin_{\y\in X}\ \norm{\x-\y}\qquad \forall\,\x\in\calH.
		\]
		It is single-valued and firmly nonexpansive \cite[Proposition 4.2]{bauschke2020correction}.
	\end{definition}
	
	\begin{definition}[Reflection and relaxed reflection]\label{def:relaxed-reflection}
		Let $X\subset\calH$ be nonempty, closed, and convex.
		The \emph{reflection} with respect to $X$ is defined by
		\[
		\calR_X \eqdef 2\proj_X-\Id.
		\]
		For $\theta\in[0,1]$, the \emph{relaxed reflection} is
		\[
		\calR_X^{\theta}\eqdef (1-\theta)\Id+\theta\,\calR_X
		= (1-\theta)\Id+\theta(2\proj_X-\Id).
		\]
		Equivalently, letting $r:=2\theta-1\in[-1,1]$, one can write
		\[
		\calR_X^{\theta} = (1+r)\proj_X - r\Id .
		\]
        Moreover, $\calR_X^{\theta}$ is nonexpansive, since it is a convex combination of $\Id$ and the nonexpansive mapping $\calR_X$. When $\theta\in(0,1)$, it is in fact $\theta$-averaged.

        For later use, we also set
    	\[
    	\calR_{Y}:=2\proj_{Y}-\Id,
    	\qquad
    	\calR_Y^{\eta}:=(1-\eta)\Id+\eta\,\calR_Y, \quad \eta\in[0,1].
    	\]
	\end{definition}
	
	\begin{remark}[A useful identity for subspaces]\label{rem:reflection-subspace-identity}
		If $X$ is a closed subspace of $\calH$, then its orthogonal complement is
		\[
		X^\perp := \{\uu\in\calH:\ \iprod{\uu}{\x}=0\ \ \forall \x\in X\}.
		\]
		In this case, $\calH=X\oplus X^\perp$ and
		\[
		\proj_{X^\perp}=\Id-\proj_X.
		\]
		Consequently,
		\[
		2\proj_X-\Id \;=\; \proj_X-\proj_{X^\perp}.
		\]
	\end{remark}
	
	\paragraph{Angles between subspaces.}
	Let $X,Y\subset\bbR^n$ be subspaces. We recall the principal angles and the Friedrichs angle; see \cite{bauschke2014rate} for more details.
	
	\begin{definition}[Principal angles]\label{def:principal-angles}
		Let $p:=\min\{\dim(X),\dim(Y)\}$.
		The principal angles $\theta_k\in\left[0,\frac{\pi}{2}\right]$, $k=1,\ldots,p$, between $X$ and $Y$ are recursively defined by
		\[
		\cos(\theta_k)
		=\max\left\{\iprod{\uu}{\vv}\ \middle|\ 
		\begin{array}{l}
			\uu\in X,\ \vv\in Y,\ \|\uu\|=\|\vv\|=1,\\
			\iprod{\uu}{\uu_j}=0,\ \iprod{\vv}{\vv_j}=0,\ j=1,\ldots,k-1
		\end{array}\right\},
		\]
		where $\uu_j\in X$ and $\vv_j\in Y$ are maximizers chosen at previous steps.
		The vectors $\uu_k,\vv_k$ need not be unique, but the angles are uniquely determined and satisfy
		$0\le \theta_1\le\cdots\le\theta_p\le \frac{\pi}{2}$.
	\end{definition}
	
	\begin{definition}[Friedrichs angle]\label{def:friedrichs-angle}
		The cosine of the Friedrichs angle between $X$ and $Y$ is defined by
		\begin{align*}
		  &\cF(X,Y) = \cos(\tF)\\
		&:= \sup\Bigl\{ |\iprod{\uu}{\vv}| \ \Big|\ 
		\uu\in X\cap(X\cap Y)^\perp,\ 
		\vv\in Y\cap(X\cap Y)^\perp,\ 
		\|\uu\|,\|\vv\|\le1
		\Bigr\}.  
		\end{align*}
		When there is no ambiguity, we write $\cF$ for $\cF(X,Y)$.
	\end{definition}
	
	\begin{proposition}[Principal angles and the Friedrichs angle]
		Let $s:=\dim(X\cap Y)$. Then $\theta_k=0$ for $k=1,\ldots,s$, and $\theta_{s+1}=\tF>0$.
	\end{proposition}
	
	\paragraph{Block-diagonal matrices.}
	For square matrices $\A_1,\ldots,\A_m$, we write
	\[
	\mathrm{blkdiag}(\A_1,\ldots,\A_m)
	:=
	\begin{bmatrix}
		\A_1 &        &        & \bm{0}\\
		& \A_2    &        &  \\
		&        & \ddots &  \\
		\bm{0}   &        &        & \A_m
	\end{bmatrix},
	\]
	i.e., the block-diagonal matrix whose diagonal blocks are $\A_1,\ldots,\A_m$ and whose off-diagonal blocks are zero.
	In particular, if $\A_i=\lambda_i \bmI_{d_i}$, then $\mathrm{blkdiag}(\lambda_1 \bmI_{d_1},\ldots,\lambda_m \bmI_{d_m})$
	is diagonal with $\lambda_i$ repeated $d_i$ times.
	
	\paragraph{Convergence rate for matrix iterations.}
	For $\A\in\bbR^{n\times n}$, we denote by $\sigma(\A)$ and $\rho(\A)$ the spectrum and spectral radius, respectively.
	We say that $\A$ converges linearly to $\A^\infty$ if there exist $r\in[0,1)$ and constants $M>0$ and $N\in\bbN$ such that
	\[
	\|\A^k-\A^\infty\|\le Mr^k \qquad \forall\,k\ge N.
	\]
	An eigenvalue $\lambda\in\sigma(\A)$ is \emph{semisimple} if
	\[
	\Ker(\A-\lambda\bmI)=\Ker\bigl((\A-\lambda\bmI)^2\bigr).
	\]
	We define the \emph{subdominant eigenvalue modulus} by
	\[
	\xi(\A):=\max\Bigl\{|\lambda|:\ 
    \lambda\in\{0\}\cup\bigl(\sigma(\A)\setminus\{1\}\bigr)
    \Bigr\}.
	\]

    \begin{fact}[{\cite[Fact~2.4]{bauschke2016optimal}}]
    Let $\A\in\bbR^{n\times n}$. Then $\lim_{k\to\infty}\A^k$ exists if and only if
    every eigenvalue $\lambda\in\sigma(\A)$ satisfies $|\lambda|<1$, except possibly $\lambda=1$, and if $1\in\sigma(\A)$, then it is semisimple.
    \end{fact}
	
	\begin{theorem}[{\cite[Theorem~2.12 and Theorem~2.15]{bauschke2016optimal}}]
		Suppose that $\A\in\bbR^{n\times n}$ is convergent to $\A^\infty\in\bbR^{n\times n}$.
		Then $\xi(\A)=\rho(\A-\A^\infty)<1$ and
		\[
		(\A-\A^\infty)^k=\A^k-\A^\infty\qquad \forall\,k\in\bbN.
		\]
		Moreover:
		\begin{itemize}
			\item[(i)] $\A$ is linearly convergent with any rate $r\in(\xi(\A),1)$;
			\item[(ii)] $\xi(\A)$ is the optimal linear rate if and only if all subdominant eigenvalues are semisimple.
		\end{itemize}
	\end{theorem}

\section{A generalized \carp}\label{sec:gcarpa}
	
	This section is devoted to the formulation and analysis of our proposed algorithm.
	Our starting point is the \carp scheme of \cite{shen2025composed}, which blends a DR-type operator with an additional projection-type correction and achieves substantial acceleration on subspace feasibility problems. We further generalize \carp by introducing relaxed reflections on both sets, leading to a more flexible family that (i) contains \carp as a special case, (ii) unifies several classical projection/reflection methods, and (iii) provides extra degrees of freedom that can improve robustness in practice.
	We first present the stationary algorithm and establish its fixed-point formulation, averagedness, and convergence.
	A non-stationary variant is discussed afterwards.
	
	\subsection{The proposed algorithm and operator viewpoint}
	
	Motivated by the spiraling behavior of the DR fixed-point sequence on locally polyhedral problems \cite{poon2020geometry}, we propose the generalized composed alternating relaxed projection algorithm (\gcarp), summarized in Algorithm \ref{alg:gcarp}.
	
        \begin{algorithm}[t]
    \caption{A generalized composed alternating relaxed projection algorithm (\gcarp)}
    \label{alg:gcarp}
    
    \textbf{Initial:} $\z^0\in\calH$; choose $\gamma\in[0,1)$,
    $\mu\in(0,2/(1+\gamma))$, and $\theta,\eta\in(0,1]$;
    
    \Repeat{convergence}{
        $\x^{k+1}=\proj_X(\z^k)$\;
        $\uu^{k+1}=\mathcal R_X^{\theta}(\z^k)=(1-\theta)\z^k+\theta(2\x^{k+1}-\z^k)$\;
        $\y^{k+1}=\proj_Y(\uu^{k+1})$\;
        $\vv^{k+1}=\mathcal R_Y^{\eta}(\uu^{k+1})=(1-\eta)\uu^{k+1}+\eta(2\y^{k+1}-\uu^{k+1})$\;
        $\z^{k+1}=(1-\mu)\z^k+\mu\big((1-\gamma)\tfrac12(\z^k+\vv^{k+1})+\gamma \y^{k+1}\big)$\;
    }
    \end{algorithm}

	\paragraph{Fixed-point operators.}
	From Algorithm~\ref{alg:gcarp}, we have
	\[
	\uu^{k+1}=\mathcal R_X^\theta \z^k,\qquad
	\y^{k+1}=\proj_Y \mathcal R_X^\theta \z^k,\qquad
	\vv^{k+1}=\mathcal R_Y^\eta \mathcal R_X^\theta \z^k.
	\]
	Hence,
	\[
	\tfrac12(\z^k+\vv^{k+1})=\tfrac12\bigl(\Id+\mathcal R_Y^\eta \mathcal R_X^\theta\bigr)\z^k,
	\qquad
	\y^{k+1}=(\proj_Y\mathcal R_X^\theta)\z^k.
	\]
	This motivates defining
		\begin{equation}\label{eq:fcarp}
		\calA^{\theta,\eta}
		:=\tfrac12\bigl(\Id+\mathcal R_Y^\eta \mathcal R_X^\theta\bigr),
		\qquad
		\calB^\theta
		:=\proj_Y\mathcal R_X^\theta,
		\end{equation}
	together with the \carp-type convex combination
	\begin{equation}\label{eq:G-def}
		\calF_\gamma^{\theta,\eta}:=(1-\gamma)\calA^{\theta,\eta}+\gamma \calB^\theta,
		\qquad
		\calF_{\mu,\gamma}^{\theta,\eta}:=(1-\mu)\Id+\mu \calF_\gamma^{\theta,\eta}.
	\end{equation}
	With these definitions, Algorithm~\ref{alg:gcarp} is exactly the fixed-point iteration
	\[
	\z^{k+1}=\calF_{\mu,\gamma}^{\theta,\eta}(\z^k).
	\]
	
	\begin{remark}[Relation with \carp and other methods]\label{rem:relations}
		We next clarify how \gcarp relates to the \carp framework of \cite{shen2025composed} and how several classical schemes are recovered as special cases. When $\theta=\eta=1$, the relaxed reflections reduce to the standard reflections, namely,
		\[
		\mathcal R_X^\theta=\mathcal R_X, \qquad \mathcal R_Y^\eta=\mathcal R_Y.
		\]
		Consequently,
		\[
		\calA^{\theta,\eta}=\tfrac12(\Id+\mathcal R_Y\mathcal R_X),
		\qquad
		\calB^\theta=\proj_Y\mathcal R_X,
		\]
		and Algorithm~\ref{alg:gcarp} reduces exactly to the \carp scheme proposed in \cite{shen2025composed}. 
		
		Under the same choice $\theta=\eta=1$, the parameter $\gamma$ continuously interpolates between two endpoint operators. On the one hand, when $\gamma=0$, the operator $\calF_\gamma^{\theta,\eta}$ becomes
		\[
		\tfrac12(\Id+\mathcal R_Y\mathcal R_X),
		\]
		which is precisely the DR operator. On the other hand, when $\gamma=1$, one obtains
		\[
		\calF_\gamma^{\theta,\eta}=\proj_Y\mathcal R_X.
		\]
		Therefore, within the \carp setting, the parameter $\gamma$ controls the extent to which the iteration moves away from the pure DR-type mechanism toward a more projection-oriented correction.
		
		Moreover, \gcarp also recovers the classical method of alternating projections. Indeed, if $\theta=\tfrac12$, then $\mathcal R_X^\theta=\proj_X$. Choosing in addition $\gamma=1$ and $\mu=1$ gives
		\[
		\z^{k+1}=\proj_Y\proj_X \z^k,
		\]
		which is exactly the standard alternating projection iteration. In this case, the resulting scheme is independent of $\eta$.
	\end{remark}

\subsection{Convergence of the \gcarp iteration}
	A key point in our design is that the desired fixed points coincide with the feasible set $X\cap Y$.
	We start from a standard identity for alternating projections.
	
	\begin{lemma}\label{lem:fix-apy}
		Assume $X\cap Y\neq\emptyset$. Then $\Fix(\proj_Y\proj_X)=X\cap Y$.
	\end{lemma}
	
	\begin{proof}
		Let $\z\in\Fix(\proj_Y\proj_X)$ and set $\x=\proj_X \z$. Then $\z=\proj_Y \x\in Y$.
		Take any $\w\in X\cap Y$.
		From $\x=\proj_X \z$ we have $\langle \z-\x,\, \w-\x\rangle\le 0$,
		and from $\z=\proj_Y \x$ we have $\langle \x-\z,\, \w-\z\rangle\le 0$, i.e.\ $\langle \z-\x,\, \w-\z\rangle\ge 0$.
		Since $\w-\x=(\w-\z)+(\z-\x)$,
		\[
		\langle \z-\x,\, \w-\x\rangle
		=\langle \z-\x,\, \w-\z\rangle+\|\z-\x\|^2\le 0.
		\]
		Together with $\langle \z-\x,\, \w-\z\rangle\ge 0$, this yields $\|\z-\x\|^2\le 0$, hence $\z=\x\in X$.
		Therefore $\z\in X\cap Y$. The reverse inclusion is immediate.
	\end{proof}
	
	\begin{lemma}\label{lem:fix-mixture}
		Let $\calT_1,\calT_2:\calH\to\calH$ be nonexpansive and assume $\Fix(\calT_1)\cap\Fix(\calT_2)\neq\emptyset$.
		Then for any $\gamma\in(0,1)$,
		\[
		\Fix\bigl((1-\gamma)\calT_1+\gamma \calT_2\bigr)=\Fix(\calT_1)\cap\Fix(\calT_2).
		\]
	\end{lemma}
	
	\begin{proof}
		Let $\p\in\Fix(\calT_1)\cap\Fix(\calT_2)$ and set $\calT=(1-\gamma)\calT_1+\gamma \calT_2$.
		If $\z\in\Fix(\calT)$, then by convexity of $\|\cdot\|^2$ and nonexpansiveness,
		\[
		\|\z-\p\|^2=\|\calT\z-\p\|^2
		\le (1-\gamma)\|\calT_1\z-\p\|^2+\gamma\|\calT_2\z-\p\|^2
		\le \|\z-\p\|^2.
		\]
		Thus equality holds throughout, which forces $\calT_1\z=\calT_2\z$ and then $\z=\calT_1\z=\calT_2\z$.
		Hence $\z\in\Fix(\calT_1)\cap\Fix(\calT_2)$.
		The reverse inclusion is obvious. 
	\end{proof}
	
	\begin{proposition}\label{prop:fix-gen}
		Assume $X\cap Y\neq\emptyset$ and $\theta,\eta\in(0,1]$.
		Then:
		\begin{itemize}
			\item[(i)] $\Fix(\calB^\theta)=X\cap Y$;
			\item[(ii)] for any $\gamma\in(0,1),			\Fix(\calF_\gamma^{\theta,\eta})=X\cap Y.
$
		\end{itemize}
	\end{proposition}
	
	\begin{proof}
		(i) If $\z\in X\cap Y$, then $\proj_X \z=\z$, so $\mathcal R_X^\theta \z=\z$, and hence $\calB^\theta \z=\proj_Y \z=\z$.
		
		Conversely, suppose $\z\in\Fix(\calB^\theta)$, i.e.\ $\z=\proj_Y(\mathcal R_X^\theta \z)$.
		Let $\x=\proj_X \z$.
		Using $\mathcal R_X^\theta \z=(1-\theta)\z+\theta(2\x-\z)=\z+2\theta(\x-\z)$ and the projector characterization
		\[
		\z=\proj_Y(\w)\quad\Longleftrightarrow\quad \w-\z\in N_Y(\z),
		\]
		we obtain $\mathcal R_X^\theta \z-\z=2\theta(\x-\z)\in N_Y(\z)$. Here $N_Y(\z):=\{\uu:\uu^\top(\y-\z)\leq0\,\,\text{for all }\y\in Y\}$ is the normal cone of $Y$ at $\z$.
        
		Since $N_Y(\z)$ is a cone and $\theta>0$, this implies $\x-\z\in N_Y(\z)$, i.e.\ $\z=\proj_Y \x=\proj_Y\proj_X \z$.
		By Lemma~\ref{lem:fix-apy}, we have that $\z\in X\cap Y$.
		
		For (ii), note that every $\z\in X\cap Y$ satisfies $\mathcal R_X^\theta \z=\z$ and $\mathcal R_Y^\eta \z=\z$,
		hence $\calA^{\theta,\eta}\z=\z$ and $\calB^\theta \z=\z$. Therefore,
		$\Fix(\calA^{\theta,\eta})\cap\Fix(\calB^\theta)$ is nonempty.
		Since $\calA^{\theta,\eta}$ and $\calB^\theta$ are nonexpansive, Lemma~\ref{lem:fix-mixture} yields
		\[
		\Fix(\calF_\gamma^{\theta,\eta})
		=\Fix\bigl((1-\gamma)\calA^{\theta,\eta}+\gamma \calB^\theta\bigr)
		=\Fix(\calA^{\theta,\eta})\cap\Fix(\calB^\theta).
		\]
		Using $\Fix(\calB^\theta)=X\cap Y$ gives $\Fix(\calF_\gamma^{\theta,\eta})=X\cap Y$. 
	\end{proof}
	
	\begin{remark}[On the role of $\gamma$]\label{rem:gamma-role}
		The case $\gamma=0$ corresponds to the pure DR-type branch $\calA^{\theta,\eta}$.
		As in the classical DR method, $\Fix(\calA^{\theta,\eta})$ may contain points outside $X\cap Y$,
		although the associated shadow sequences can still converge to a feasible point.
		In this work we aim to obtain a fixed-point limit lying directly in $X\cap Y$.
		For this reason we focus on $\gamma>0$, for which Proposition~\ref{prop:fix-gen} guarantees
		$\Fix(\calF_\gamma^{\theta,\eta})=X\cap Y$.
	\end{remark}
	
	\begin{theorem}\label{th:avg-and-conv-gcarp}
		Let $X,Y\subset\calH$ be nonempty closed convex sets and let $\theta,\eta\in(0,1]$.
		For $\gamma\in[0,1)$ define $\calF_\gamma^{\theta,\eta},\calF_{\mu,\gamma}^{\theta,\eta},\calA^{\theta,\eta},\calB^\theta$ as in \eqref{eq:fcarp} and \eqref{eq:G-def}.
		\begin{enumerate}
			\item[(i)] \textbf{(Nonexpansiveness and averagedness).}
			The relaxed reflections $\mathcal R_X^\theta$ and $\mathcal R_Y^\eta$ are nonexpansive.
			Consequently, $\calA^{\theta,\eta}$ is firmly nonexpansive ($\tfrac12$-averaged) and $\calB^\theta$ is nonexpansive. Moreover, $\calF_\gamma^{\theta,\eta}$ is $\tfrac{1+\gamma}{2}$-averaged, and for any $			\mu\in(0,\frac{2}{1+\gamma})$, the relaxed operator $\calF_{\mu,\gamma}^{\theta,\eta}$ is $\beta$-averaged with $			\beta=\frac{(1+\gamma)\mu}{2}\in(0,1)$.
			
			\item[(ii)] \textbf{(Global convergence).}
			Assume $X\cap Y\neq\emptyset$ and let $\gamma\in(0,1)$ and
			$\mu\in\bigl(0,\frac{2}{1+\gamma}\bigr)$.
			Then the sequence $\{\z^k\}$ generated by Algorithm~\ref{alg:gcarp}, i.e., $\z^{k+1}=\calF_{\mu,\gamma}^{\theta,\eta}(\z^k)$, converges weakly to some $\z^\star\in X\cap Y$: $\z^k \rightharpoonup \z^\star\in X\cap Y$.
			If, in addition, $\calH$ is finite-dimensional, then $\z^k\to \z^\star$ and consequently
			$\x^k\to \z^\star$ and $\y^k\to \z^\star$.
		\end{enumerate}
	\end{theorem}
	
	\begin{proof}
		(i) Since $\proj_X$ and $\proj_Y$ are firmly nonexpansive, the reflections
		$\mathcal R_X=2\proj_X-\Id$ and $\mathcal R_Y=2\proj_Y-\Id$ are nonexpansive.
		Because $\mathcal R_X^\theta=(1-\theta)\Id+\theta\mathcal R_X$ and
		$\mathcal R_Y^\eta=(1-\eta)\Id+\eta\mathcal R_Y$ are convex combinations of nonexpansive maps,
		they are nonexpansive as well. Hence $\mathcal R_Y^\eta\mathcal R_X^\theta$ is nonexpansive and
		\[
		\calA^{\theta,\eta}=\tfrac12(\Id+\mathcal R_Y^\eta\mathcal R_X^\theta)
		\]
		is firmly nonexpansive. Also $\proj_Y$ is nonexpansive, so $\calB^\theta=\proj_Y\mathcal R_X^\theta$
		is nonexpansive.
		
		Writing $\calA^{\theta,\eta}$ as $\tfrac12$-averaged and $\calB^\theta$ as $1$-averaged,
		the convex-mixture rule for averaged operators \cite[Proposition~2.2]{CombettesYamada15} yields that $\calF_\gamma^{\theta,\eta}$ is averaged with
		parameter
		\[
		(1-\gamma)\cdot\tfrac12+\gamma\cdot 1=\tfrac{1+\gamma}{2}.
		\]
		Finally, if $\calF_\gamma^{\theta,\eta}$ is $\alpha$-averaged with $\alpha=\tfrac{1+\gamma}{2}$, then
		$\calF_{\mu,\gamma}^{\theta,\eta}=(1-\mu)\Id+\mu \calF_\gamma^{\theta,\eta}$ is $(\mu\alpha)$-averaged, i.e.,
		$\beta=\mu\alpha=\tfrac{(1+\gamma)\mu}{2}\in(0,1)$ under $\mu\in(0,\frac{2}{1+\gamma})$.
		
		(ii) Assume $X\cap Y\neq\emptyset$. By Proposition~\ref{prop:fix-gen},
		\[
		\Fix\bigl(\calF_{\mu,\gamma}^{\theta,\eta}\bigr)=\Fix\bigl(\calF_{\gamma}^{\theta,\eta}\bigr)=X\cap Y.
		\]
		Since $\calF_{\mu,\gamma}^{\theta,\eta}$ is averaged by part (i), the Krasnosel'ski\u{\i}--Mann theorem for averaged
		operators (e.g., \cite[Theorem~5.14]{bauschke2020correction}) yields weak convergence
		$\z^k\rightharpoonup \z^\star\in X\cap Y$.
		
		If $\calH$ is finite-dimensional, weak convergence implies strong convergence.
		By continuity of $\proj_X$, $\proj_Y$, and since $\z^\star\in X\cap Y$ implies
		$\mathcal R_X^\theta(\z^\star)=\z^\star$, we have
		$\x^{k+1}=\proj_X(\z^k)\to \z^\star$ and
		$\y^{k+1}=\proj_Y(\mathcal R_X^\theta \z^k)\to \z^\star$. 
	\end{proof}
	
	\begin{remark}[Asymptotic regularity and residuals]\label{rem:asymp-reg}
		Assume $X\cap Y\neq\emptyset$ and set $\calT:=\calF_{\mu,\gamma}^{\theta,\eta}$.
		Since $\calT$ is averaged, it is strongly quasi-nonexpansive; in particular, for every
		$\z\in\Fix(\calT)=X\cap Y$ one has the descent inequality
		\[
		\|\calT \x-\z\|^2 \le \|\x-\z\|^2-\frac{1-\beta}{\beta}\,\|(\Id-\calT)\x\|^2,
		\]
		see, e.g., \cite[Proposition~4.25]{bauschke2020correction}, where $\beta\in(0,1)$ denotes the averagedness constant.

		Applying this with $\x=\z^k$ yields that $\{\z^k\}$ is Fej\'er monotone with respect to $X\cap Y$ and that
		\[
		\sum_{k=0}^{\infty}\|\z^{k+1}-\z^k\|^2=\sum_{k=0}^{\infty}\|(\Id-\calT)\z^k\|^2<\infty,
		\]
		hence $\|\z^{k+1}-\z^k\|\to0$.
		
		Moreover, since $\calT$ is the relaxed map $\calT=(1-\mu)\Id+\mu \calF_\gamma^{\theta,\eta}$ with $\mu>0$, then
		\[
		\|\calF_\gamma^{\theta,\eta}(\z^k)-\z^k\|
		=\frac{1}{\mu}\,\|\calT\z^k-\z^k\|
		=\frac{1}{\mu}\,\|\z^{k+1}-\z^k\|\ \longrightarrow\ 0.
		\]
		These properties justify stopping criteria based on fixed-point residuals such as
		$\|\z^{k+1}-\z^k\|$ or equivalently $\|\calF_\gamma^{\theta,\eta}(\z^k)-\z^k\|$.
	\end{remark}

\subsection{Optimal linear convergence rate of \gcarp}\label{sec:gcarp-opt}
	
	In this subsection, we derive explicit spectral formulas and optimal-rate consequences for \gcarp
	(Algorithm~\ref{alg:gcarp}) in the benchmark setting where $X$ and $Y$ are linear subspaces of $\mathbb{R}^n$.
	
	Let $X,Y\subset\mathbb{R}^n$ be subspaces with $\dim(X)=p\le q=\dim(Y)$.
	Denote by $\phi_1,\ldots,\phi_p\in[0,\pi/2]$, the principal angles between $X$ and $Y$, and set
	\[
	c_i:=\cos\phi_i,\qquad s_i:=\sin\phi_i,\qquad t_i:=s_i^2\in[0,1]\qquad (i=1,\ldots,p).
	\]
	Let $s:=\dim(X\cap Y)$, so that $\phi_1=\cdots=\phi_s=0$ and $\phi_{s+1}=\phi_F>0$ is the Friedrichs angle.
	
	Consequently, it suffices to analyze $p+q<n$.
	
	By \cite{bauschke2016optimal}, there exists an orthogonal matrix $\D\in\mathbb{R}^{n\times n}$ such that 
	\begin{equation}\label{eq:PX-PY-principal-gcarp2}
		\bmP_X
		= \D
		\begin{bmatrix}
			\bmI_p & \bm0 & \bm0 & \bm0\\
			\bm0 & \bm0_p & \bm0 & \bm0\\
			\bm0 & \bm0 & \bm0_{q-p} & \bm0\\
			\bm0 & \bm0 & \bm0 & \bm0_{n-p-q}
		\end{bmatrix}
		\D^\top,\quad
		\bmP_Y
		= \D
		\begin{bmatrix}
			\C^2 & \C\bmS & \bm0 & \bm0\\
			\C\bmS & \bmS^2 & \bm0 & \bm0\\
			\bm0 & \bm0 & \bmI_{q-p} & \bm0\\
			\bm0 & \bm0 & \bm0 & \bm0_{n-p-q}
		\end{bmatrix}
		\D^\top,
	\end{equation}
	where $\C=\mathrm{diag}(c_i)_{i=1}^p$,  $\bmS=\mathrm{diag}(s_i)_{i=1}^p$, and $\bmP_X,\bmP_Y$ are the matrix representations of the projection operators $\proj_X,\proj_Y$ in the orthogonal basis induced by $\bmD$.
	
	In the following analysis, we will use the abbreviations
	\[
	a:=1-2\theta\in[-1,1),\qquad
	\kappa:=\gamma+(1-\gamma)\eta=\eta+\gamma(1-\eta)\in[\eta,1).
	\]
	
	\begin{theorem}[Spectral decomposition and linear rate of \gcarp]\label{th:FGCARPgammamu-rewrite}
		Let $X,Y\subset\mathbb{R}^n$ be subspaces with $p=\dim(X)\le q=\dim(Y)$ and $p+q< n$.
		Fix $\gamma\in[0,1)$, $\theta,\eta\in(0,1]$, and $\mu\in\bigl(0,\frac{2}{1+\gamma}\bigr)$.
		Consider the fixed-point iteration
		\[
		\z^{k+1}=\mathcal{F}_{\mu,\gamma}^{\theta,\eta}\z^k.
		\]
		
		\smallskip
		\noindent\textbf{(i) Convergence.}
		$\mathcal{F}_{\gamma}^{\theta,\eta}$ is $\frac{1+\gamma}{2}$-averaged and
		$\mathcal{F}_{\mu,\gamma}^{\theta,\eta}$ is $\alpha$-averaged with
		\[
		\alpha=\frac{(1+\gamma)\mu}{2}\in(0,1).
		\]
		Hence, for any $\z^0\in\mathbb{R}^n$, the sequence $\{\z^k\}$ converges to a point in
		$\Fix(\mathcal{F}_{\mu,\gamma}^{\theta,\eta})$.
		
		\smallskip
		\noindent\textbf{(ii) Block diagonalization and explicit eigenvalues.}
		In the principal-angle coordinates \eqref{eq:PX-PY-principal-gcarp2},
		\[
		\mathcal{F}_{\mu,\gamma}^{\theta,\eta}
		=
		\D\;\mathrm{blkdiag}\bigl(\bmM_1,\ldots,\bmM_p,\;\lambda_{q-p}\bmI_{q-p},\;\lambda_{n-p-q}\bmI_{n-p-q}\bigr)\;\D^\top,
		\]
		where for each $i=1,\ldots,p$,
		\begin{equation}\label{eq:Mi-gcarp2}
			\bmM_i
			=
			\begin{bmatrix}
				1-\mu\kappa t_i & \mu a\kappa c_i s_i\\[1mm]
				\mu \kappa c_i s_i & 1-\mu\bigl((1+\gamma)\theta+a\kappa c_i^2\bigr)
			\end{bmatrix},
		\end{equation}
		and the scalar blocks are
		\begin{equation}\label{eq:scalar-gcarp2}
			\lambda_{q-p}=1-\mu(1+\gamma)\theta,\qquad
			\lambda_{n-p-q}=(1-\mu)+\mu(1-\gamma)\bigl(1-\theta-\eta+2\theta\eta\bigr).
		\end{equation}
		The two eigenvalues of $\bmM_i$ are
		\begin{equation}\label{eq:eigs-Mi-gcarp2}
			\lambda_{i,\pm}
			=
			1-\frac{\mu}{2}\Bigl((1+\gamma)\theta+\kappa\bigl(1-2\theta c_i^2\bigr)\Bigr)
			\pm
			\frac{\mu}{2}\sqrt{\Delta_i},
		\end{equation}
		with discriminant
		\begin{equation}\label{eq:Delta-i-vertex}
			\Delta_i
			=
			\Bigl((1+\gamma)\theta+a\kappa-2(1-\theta)\kappa t_i\Bigr)^2
			+4a\kappa^2 t_i(1-t_i)
			=
			4\kappa^2\theta^2(t_i-\tau)^2+\Delta_{\min},
		\end{equation}
		where
		\begin{equation}\label{eq:tau-deltamin}
			\tau:=\frac{(1+\gamma)(1-\theta)-a\kappa}{2\kappa\theta},
			\qquad
			\Delta_{\min}=(1-\gamma^2)(2\theta-1)(1-2\eta).
		\end{equation}
		
		\smallskip
		\noindent\textbf{(iii) Linear rate and optimal rate in the sense of $\xi(\cdot)$.}
		Let $\A:=\mathcal{F}_{\mu,\gamma}^{\theta,\eta}$ and let $\A^\infty:=\lim_{k\to\infty}\A^k$.
		Then $\A^\infty$ exists and the convergence is linear.
		Moreover, the subdominant-eigenvalue modulus (as defined in Section~\ref{sec:maths_bg}) satisfies
		\[
		\xi(\A)=
		\max\Bigl\{
		\max_{1\le i\le p}\max\{|\lambda_{i,+}|,|\lambda_{i,-}|\},\;
		|\lambda_{q-p}|,\;
		|\lambda_{n-p-q}|
		\Bigr\}
		\in[0,1).
		\]
		Consequently, $\A$ is linearly convergent with any rate $r\in(\xi(\A),1)$.
		If, in addition, all eigenvalues of modulus $\xi(\A)$ are semisimple, then $\xi(\A)$ is the optimal linear rate in the sense of Theorem~2.12--2.15 of \cite{bauschke2016optimal} recalled in Section~\ref{sec:maths_bg}.
		
		\smallskip
		\noindent\textbf{(iv) Specialization to $X\cap Y$ for $\gamma\in(0,1)$.}
		If $\gamma\in(0,1)$, then $\Fix(\mathcal{F}_{\mu,\gamma}^{\theta,\eta})=\Fix(\mathcal{F}_{\gamma}^{\theta,\eta})=X\cap Y$
		\[
		\A^\infty=\bmP_{X\cap Y},
		\,\text{hence}\,\,
		\z^k \to \bmP_{X\cap Y}\z^0 \ \text{linearly.}
		\]
        Here $\bmP_{X\cap Y}$ is the matrix representation of the projection $\proj_{X\cap Y}$.
	
	\end{theorem}

	\begin{proof}
		(i) The conclusion follows from Theorem \ref{th:avg-and-conv-gcarp}.
	
		(ii)
		Insert the matrix representation \eqref{eq:PX-PY-principal-gcarp2} into
		$\mathcal R_X^\theta=(1-2\theta)\Id+2\theta \proj_X$ and $\mathcal R_Y^\eta=(1-2\eta)\Id+2\eta \proj_Y$,
		then expand \eqref{eq:G-def} in the orthogonal basis induced by $\D$.
		On each principal-angle plane, the action is a $2\times2$ linear map, yielding \eqref{eq:Mi-gcarp2}.
		On the residual subspaces of dimensions $(q-p)$ and $(n-p-q)$, the action is scalar, giving \eqref{eq:scalar-gcarp2}.
		This proves (ii).
		
		(iii)
		For $\bmM_i=\begin{bmatrix}\alpha & \beta\\ \delta & \varepsilon\end{bmatrix}$,
		$\lambda_{i,\pm}=\frac{\alpha+\varepsilon}{2}\pm\frac12\sqrt{(\alpha-\varepsilon)^2+4\beta\delta}$.
		Substituting \eqref{eq:Mi-gcarp2} yields \eqref{eq:eigs-Mi-gcarp2} with
		\[
		\Delta_i=(\alpha-\varepsilon)^2+4\beta\delta
		=
		\Bigl((1+\gamma)\theta+a\kappa-2(1-\theta)\kappa t_i\Bigr)^2
		+4a\kappa^2 t_i(1-t_i).
		\]
		Completing the square gives the vertex form \eqref{eq:Delta-i-vertex} and \eqref{eq:tau-deltamin}.
		
		Since $\A=\mathcal{F}_{\mu,\gamma}^{\theta,\eta}$ is a linear averaged operator on $\mathbb{R}^n$, the limit $\A^\infty=\lim_{k\to\infty}\A^k$ exists
		and the convergence rate is linear, as established in Fact~2.4 and Theorems~2.12--2.15 of \cite{bauschke2016optimal}. By the block decomposition in (ii), the set of eigenvalues of $\A$ is the union of
		$\{\lambda_{i,\pm}\}_{i=1}^p$ and the scalar eigenvalues \eqref{eq:scalar-gcarp2}.
		The quantity $\xi(\A)$ is therefore given by the stated maximum, excluding the eigenvalue $1$ coming from the intersection component,
		and the linear rate statement in (iii) follows from Theorem~2.12--2.15 of \cite{bauschke2016optimal}.
		
		(iv)
		For $\gamma\in(0,1)$, Proposition~\ref{prop:fix-gen} gives $\Fix(\A)=X\cap Y$.
		In the principal-angle basis, the only eigenvalue equal to $1$ occurs precisely on the intersection directions corresponding to $\phi_i=0$, while all other blocks have spectral radius strictly smaller than 1 because $\alpha\in(0,1)$.
		Hence $\A^k$ converges to the orthogonal projector onto the $1$-eigenspace, which is exactly $X\cap Y$,
		so $\A^\infty=\bmP_{X\cap Y}$ and (iv) follows. 
	\end{proof}
	
	\begin{remark}
		The case $p+q> n$ can be reduced to this setting by a standard lifting \cite{pong2016eigenvalue}.
		Indeed, choose an integer $r\ge 1$ such that $n':=n+r>p+q$, and define
		\[
		X':=X\times\{0_r\}\subset\mathbb{R}^{n'},\qquad
		Y':=Y\times\{0_r\}\subset\mathbb{R}^{n'}.
		\]
		Then, $\dim(X')=p$, $\dim(Y')=q$, and $p+q<n'$. Moreover, the principal angles between $X'$ and $Y'$
		coincide with those between $X$ and $Y$, and the lifted projectors satisfy
		\[
		\bmP_{X'}=
		\begin{bmatrix}
			\bmP_X & \bm0\\
			\bm0 & \bm0_r
		\end{bmatrix},
		\qquad
		\bmP_{Y'}=
		\begin{bmatrix}
			\bmP_Y & \bm0\\
			\bm0 & \bm0_r
		\end{bmatrix}.
		\]
		Hence, any operator constructed from $\bmP_X,\bmP_Y$ and their reflections/relaxations inherit the same block-diagonal embedding,
		so the principal-angle-dependent blocks, and therefore the linear rate are unchanged by the lifting.
	\end{remark}

    The spectral characterization in Theorem~\ref{th:FGCARPgammamu-rewrite} shows that, for subspace feasibility,
	the linear convergence factor of \gcarp is determined by the spectral radii of the $2\times2$ principal-angle
	blocks. In practice, however, directly minimizing the worst-case spectral radius over all principal angles
	$\{t_i=\sin^2(\phi_i)\}$ is cumbersome, since it leads to a nontrivial minimax problem in $\gamma$.
	
	In the symmetric setting $\mu=1$, $p=q$, and $p+q=n$, the iteration has only these $2\times2$ blocks and the
	principal angles fully govern the dynamics. Moreover, when $\theta,\eta\in[1/2,1]$, the worst-case block rate
	as a function of $t$ exhibits a decrease--then--increase pattern, implying that the maximum over
	$\{t_i\}$ is attained at the two extreme angles, i.e., the Friedrichs angle and the largest principal angle.
	This observation allows us to reduce the tuning of $\gamma$ to comparing a small set of explicit candidates
	obtained from a critical-damping condition $\Delta(t)=0$. The next theorem formalizes this endpoint-worst
	property and provides an explicit minimax recipe for selecting $\gamma$.
	
	\begin{theorem}\label{th:gcarp-minimax-gamma}
		Let $X,Y\subset\mathbb{R}^n$ be two subspaces with $p=\dim(X),q=\dim(Y)$, $p=q$ and $p+q=n$.
		Assume $\mu=1$ and $\theta,\eta\in[1/2,1]$.
		Let $\phi_1,\ldots,\phi_p\in(0,\pi/2]$ be the nonzero principal angles between $X$ and $Y$, and set
		$t_i:=\sin^2(\phi_i)\in(0,1]$.
		Denote by $\phi_F$ the Friedrichs angle and by $\phi_p$ the largest principal angle, and write
		$t_F:=\sin^2(\phi_F)$ and $t_p:=\sin^2(\phi_p)$.
		
		For each $t\in(0,1)$ and $\gamma\in[0,1)$, let $\bmM(t,\gamma)$ be the $2\times 2$ principal-angle block of $\mathcal{F}_{\mu,\gamma}^{\theta,\eta}$ obtained from \eqref{eq:Mi-gcarp2} with $\mu=1$,
		and let $\rho(\bmM(t,\gamma))$ denote its spectral radius. Define
		\begin{equation}\label{eq:wpsi2-minimax}
        \begin{aligned}
        \psi(t)&:=t+\sqrt{2\theta-1}\sqrt{t(1-t)},\\
        w_\pm(t)&:=(2\theta-1)+2(1-\theta)t\pm 2\sqrt{2\theta-1}\sqrt{t(1-t)}.
        \end{aligned}
        \end{equation}
		Moreover, set
		\[
		\kappa(\gamma):=\eta+(1-\eta)\gamma,\qquad a:=1-2\theta\le 0.
		\]
		
		\smallskip
		\noindent\textbf{(i) Endpoint-worst property.}
		For every fixed $\gamma\in[0,1)$,
		\begin{equation}\label{eq:endpoint-worst2-minimax}
			\max_{1\le i\le p}\rho\bigl(\bmM(t_i,\gamma)\bigr)
			=\max\bigl\{\rho(\bmM(t_F,\gamma)),\ \rho(\bmM(t_p,\gamma))\bigr\}.
		\end{equation}
		
		\smallskip
		\noindent\textbf{(ii) Critical damping candidates.}
		Let $\Delta(t,\gamma)$ be the discriminant of the characteristic polynomial of $\bmM(t,\gamma)$, i.e., the quantity in \eqref{eq:Delta-i-vertex} with $t_i=t$.
		If one enforces the critical damping condition $\Delta(t,\gamma)=0$ for a given $t\in(0,1)$,
		then necessarily
		\begin{equation}\label{eq:gamma-star2-minimax}
			\gamma=\gamma_\pm(t):=\frac{\eta\,w_\pm(t)-\theta}{\theta-(1-\eta)\,w_\pm(t)},
			\qquad\text{whenever }\gamma_\pm(t)\in[0,1).
		\end{equation}
		At such a parameter, the two eigenvalues of $\bmM(t,\gamma)$ coincide and satisfy
		\begin{equation}\label{eq:rho-star2-minimax}
        \begin{aligned}
            \lambda^\star(t)&=1-\kappa^\star(t)\,\psi(t),\\
            \kappa^\star(t)&:=\kappa(\gamma_\pm(t))=\eta+(1-\eta)\gamma_\pm(t)=\frac{(1+\gamma_\pm(t))\theta}{w_\pm(t)}.
        \end{aligned}
		\end{equation}
		
		\smallskip
		\noindent\textbf{(iii) Minimax recipe.}
		Consider the minimax problem
		\[
		\min_{\gamma\in[0,1)}\ \max_{1\le i\le p}\rho\bigl(\bmM(t_i,\gamma)\bigr).
		\]
		By \eqref{eq:endpoint-worst2-minimax}, it reduces to minimizing
		$\max\{\rho(\bmM(t_F,\gamma)),\rho(\bmM(t_p,\gamma))\}$.
		A practical minimax choice can therefore be obtained by evaluating
		the endpoint critical-damping candidates $\gamma_\pm(t_F)$ and $\gamma_\pm(t_p)$, whenever they lie in $[0,1)$, and selecting the one that gives the smallest value of
		$\max\{\rho(\bmM(t_F,\gamma)),\rho(\bmM(t_p,\gamma))\}$.
		
		\smallskip
		\noindent\textbf{(iv) Reduction to \carp.}
		When $\theta=\eta=1$, \eqref{eq:gamma-star2-minimax}--\eqref{eq:rho-star2-minimax}
		reduce to the \carp formulas
		$\gamma^\star(t)=2\sqrt{t(1-t)}=2cs$ and $\rho^\star(t)=|c^2-cs|$ in \cite{shen2025composed}.
	\end{theorem}
	
	\begin{proof}
		(i)	Fix $\gamma\in[0,1)$ and a principal-angle parameter $t\in(0,1)$.
		Write $c:=\sqrt{1-t}$ and $s:=\sqrt{t}$ so that $cs=\sqrt{t(1-t)}$.
		Let $\lambda_\pm(t,\gamma)$ be the eigenvalues of $M(t,\gamma)$ and let $\Delta(t,\gamma)$
		be their discriminant.	
	
		If $\Delta(t,\gamma)<0$, then $\lambda_-(t,\gamma)=\overline{\lambda_+(t,\gamma)}$ and hence
		\begin{equation}\label{eq:rho2-det-minimax}
			\rho(\bmM(t,\gamma))^2=|\lambda_+(t,\gamma)|^2=\det \bmM(t,\gamma).
		\end{equation}
		Using the explicit block \eqref{eq:Mi-gcarp2} with $\mu=1$ and the abbreviations
		$\kappa=\kappa(\gamma)$ and $a=1-2\theta$, a direct simplification yields
		\begin{equation}\label{eq:det-affine-minimax}
			\det \bmM(t,\gamma)=(1-\gamma)\bigl(b-\theta\kappa\,t\bigr),
			\qquad
			b:=1-\theta-\eta+2\theta\eta.
		\end{equation}
		Since $\theta>0$ and $\kappa>0$, the map $t\mapsto\det M(t,\gamma)$ is affine and strictly decreasing on $(0,1)$.
		Therefore, in the complex-eigenvalue regime, $t\mapsto \rho(\bmM(t,\gamma))$ is strictly decreasing.
		
		If $\Delta(t,\gamma)\ge 0$, then $\lambda_\pm(t,\gamma)\in\mathbb{R}$ and
		\[
		\rho(\bmM(t,\gamma))=\max\{|\lambda_+(t,\gamma)|,|\lambda_-(t,\gamma)|\}.
		\]
		
		We now show that, on the boundary intervals where $\Delta(t,\gamma)\ge 0$ holds,
		the extremal eigenvalue branches are monotone.
	
		Under $\theta,\eta\in[1/2,1]$, we have $a=1-2\theta\le 0$ and
		$\Delta_{\min}=(1-\gamma^2)(2\theta-1)(1-2\eta)\le 0$ in \eqref{eq:Delta-i-vertex}.
		Hence, the vertex form \eqref{eq:Delta-i-vertex} implies that $\Delta(t,\gamma)\ge 0$
		holds only near the endpoints $t=0$ and $t=1$, unless $\Delta_{\min}=0$, in which case $\Delta\equiv 0$.
		More precisely, when $\Delta_{\min}<0$, there exists $t_- \le t_+$ such that
		\[
		\Delta(t,\gamma)\ge 0\quad\Longleftrightarrow\quad t\in[0,t_-]\cup[t_+,1].
		\]
		On $[0,t_-]$, the square-root term $\sqrt{\Delta(t,\gamma)}$ decreases as $t$ increases, since $\Delta$ is a convex parabola with vertex at $\tau$, while the trace term
		$\tr \bmM(t,\gamma)$ is affine and strictly decreasing in $t$ due to the term $-\kappa\theta t$.
		Consequently, the upper eigenvalue branch $\lambda_+(t,\gamma)
		=\frac12\tr \bmM(t,\gamma)+\frac12\sqrt{\Delta(t,\gamma)}$ is strictly decreasing on $[0,t_-]$,
		and hence $\rho(M(t,\gamma))=\lambda_+(t,\gamma)$ decreases there.
		
		On $[t_+,1]$, the square-root term $\sqrt{\Delta(t,\gamma)}$ increases in $t$ whereas
		$\tr \bmM(t,\gamma)$ keeps decreasing linearly. This implies that $\lambda_-(t,\gamma)
		=\frac12\tr \bmM(t,\gamma)-\frac12\sqrt{\Delta(t,\gamma)}$ decreases with $t$ and becomes negative
		for $t$ sufficiently close to $1$; therefore $-\lambda_-(t,\gamma)$ increases on $[t_+,1]$.
		As a result, $\rho(\bmM(t,\gamma))=\max\{\lambda_+(t,\gamma),-\lambda_-(t,\gamma)\}$ is increasing on $[t_+,1]$.
		
		Combining this with Step~1, i.e., strict decrease in the complex regime, we conclude that for every fixed $\gamma$,
		the function $t\mapsto \rho(\bmM(t,\gamma))$ decreases for small $t$ and increases for large $t$.
		Hence, on any interval $[t_F,t_p]$ containing all $\{t_i\}$, the maximum is attained at an endpoint,
		which proves \eqref{eq:endpoint-worst2-minimax}.
		
		(ii) Fix $t\in(0,1)$ and enforce $\Delta(t,\gamma)=0$.
		Using the explicit discriminant in \eqref{eq:Delta-i-vertex} with $t_i=t$ and $\mu=1$,
		this condition is equivalent to
		\[
		\Bigl((1+\gamma)\theta+a\kappa-2(1-\theta)\kappa t\Bigr)^2
		=
		4(2\theta-1)\kappa^2 t(1-t).
		\]
		Taking square roots and using $-a=2\theta-1$ yields
		\[
		\frac{(1+\gamma)\theta}{\kappa}
		=
		(2\theta-1)+2(1-\theta)t\pm 2\sqrt{2\theta-1}\sqrt{t(1-t)}=w_\pm(t).
		\]
		Recalling $\kappa=\eta+(1-\eta)\gamma$, rearranging gives \eqref{eq:gamma-star2-minimax}.
		
		At $\Delta(t,\gamma)=0$, the eigenvalues coalesce to $\lambda^\star(t)=\frac12\tr \bmM(t,\gamma)$.
		Substituting $\frac{(1+\gamma)\theta}{\kappa}=w_\pm(t)$ into $\tr \bmM(t,\gamma)$
		yields \eqref{eq:rho-star2-minimax}, where $\psi(t)$ is defined in \eqref{eq:wpsi2-minimax}.
		
		(iii) By \eqref{eq:endpoint-worst2-minimax}, the worst-case rate over $\{t_i\}$ equals the endpoint maximum
		$$\max\{\rho(\bmM(t_F,\gamma)),\rho(\bmM(t_p,\gamma))\}.$$ Therefore, the minimax problem reduces to
		choosing $\gamma$ to balance or reduce these two endpoint rates.
		The critical-damping values $\gamma_\pm(t_F)$ and $\gamma_\pm(t_p)$ provide explicit endpoint candidates;
		comparing their endpoint maxima yields the stated recipe.
		
		(iv) If $\theta=\eta=1$, then $\sqrt{2\theta-1}=1$ and $w_\pm(t)=1\pm 2\sqrt{t(1-t)}$.
		Since $\eta=1$ implies $\kappa\equiv 1$, \eqref{eq:gamma-star2-minimax} gives
		$\gamma^\star(t)=2\sqrt{t(1-t)}$, and \eqref{eq:rho-star2-minimax} gives
		$\lambda^\star(t)=1-(t+\sqrt{t(1-t)})=c^2-cs$, hence $\rho^\star(t)=|c^2-cs|$. 
	\end{proof}

    \subsection{A fully non-stationary scheme (ns-\gcarp\ with varying $\gamma_k,\theta_k,\eta_k$)}
	\label{subsec:nsfull}
	
	The stationary \gcarp iteration $\z^{k+1}=\mathcal{F}_{\mu,\gamma}^{\theta,\eta}(\z^k)$ already enjoys global convergence
	for closed convex feasibility problems, and in the subspace case we can further characterize its linear rate via
	principal angles.
	In practice, however, the best-performing parameters may depend on the local geometry and can vary along the iterates.
	This motivates a fully non-stationary variant, Alogrithm \ref{alg:nsfull}, in the spirit of non-stationary averaged-operator schemes; see, e.g., \cite{liang2016convergence}.
	\begin{algorithm}[t]
\caption{Fully non-stationary generalized scheme (ns-\gcarp)}
\label{alg:nsfull}
\normalsize
\textbf{Initial:} $k=0$, $\z^0\in\calH$; choose
$\gamma_k\in[\gamma_{\min},\gamma_{\max}]\subset(0,1)$, $\theta_k\in[\theta_{\min},\theta_{\max}]\subset(0,1]$,
$\eta_k\in[\eta_{\min},\eta_{\max}]\subset(0,1]$;
$\mu\in\bigl(0,\frac{2}{1+\gamma_{\max}}\bigr)$;
constants $c_\gamma,c_\theta,c_\eta,\delta>0$\;

\Repeat{convergence}{
    $\x^{k+1}=\proj_X(\z^k)$\;
    $\uu^{k+1}=\mathcal R_X^{\theta_k}(\z^k)=(1-\theta_k)\z^k+\theta_k(2\x^{k+1}-\z^k)$\;
    $\y^{k+1}=\proj_Y(\uu^{k+1})$\;
    $\vv^{k+1}=\mathcal R_Y^{\eta_k}(\uu^{k+1})=(1-\eta_k)\uu^{k+1}+\eta_k(2\y^{k+1}-\uu^{k+1})$\;
    $\z^{k+1}=(1-\mu)\z^k+\mu\Big((1-\gamma_k)\tfrac12(\z^k+\vv^{k+1})+\gamma_k \y^{k+1}\Big)$\;
    Update $(\gamma_{k+1},\theta_{k+1},\eta_{k+1})$ so that
    \[
	|\gamma_{k+1}-\gamma_k|\le \frac{c_\gamma}{(k+1)^{2+\delta}},\,\,
	|\theta_{k+1}-\theta_k|\le \frac{c_\theta}{(k+1)^{2+\delta}},\,\,
	|\eta_{k+1}-\eta_k|\le \frac{c_\eta}{(k+1)^{2+\delta}},
					\]
    and clip back to the above intervals if needed\;
}
\end{algorithm}
	
	\begin{remark}[How to enforce the increment conditions]\label{rmk:nsfull-updates}
		Algorithm~\ref{alg:nsfull} only requires that successive parameter changes are summably small; the specific update rule
		is otherwise flexible. A convenient way to guarantee the increment bounds is to proceed in two stages:
		first generate trial parameters $(\tilde\gamma_{k+1},\tilde\theta_{k+1},\tilde\eta_{k+1})$ by any practical
		heuristic, then apply a diminishing damping step and clip.
		
		For instance, let $\alpha_k:=\frac{1}{(k+1)^{2+\delta}}$ and set
		\begin{align*}
			\gamma_{k+1}
			&=\Pi_{[\gamma_{\min},\gamma_{\max}]}\big((1-\alpha_k)\gamma_k+\alpha_k\tilde\gamma_{k+1}\big),\\
			\theta_{k+1}
			&=\Pi_{[\theta_{\min},\theta_{\max}]}\big((1-\alpha_k)\theta_k+\alpha_k\tilde\theta_{k+1}\big),\\
			\eta_{k+1}
			&=\Pi_{[\eta_{\min},\eta_{\max}]}\big((1-\alpha_k)\eta_k+\alpha_k\tilde\eta_{k+1}\big),
		\end{align*}
		where $\Pi_{[a,b]}$ denotes the Euclidean projection onto the interval $[a,b]$.
		Then
		\begin{align*}
			|\gamma_{k+1}-\gamma_k|
			&\le \alpha_k(\gamma_{\max}-\gamma_{\min}),\\
			|\theta_{k+1}-\theta_k|
			&\le \alpha_k(\theta_{\max}-\theta_{\min}),\\
			|\eta_{k+1}-\eta_k|
			&\le \alpha_k(\eta_{\max}-\eta_{\min}),
		\end{align*}
		so the increment conditions of Algorithm~\ref{alg:nsfull} hold.
		
		As an example of a trial rule, one may use Armijo--Goldstein-like multiplicative updates driven by
		$\rho_k=\|\z^{k+1}-\z^k\|/\|\z^k-\z^{k-1}\|$:
		\begin{align*}
			\tilde\gamma_{k+1}&=
			\begin{cases}
				\gamma_k\cdot \varsigma_\gamma, & \rho_k<c_1,\\
				\min\{\gamma_k/\varsigma_\gamma,\gamma_{\max}\}, & \text{otherwise},
			\end{cases}\\
			\tilde\theta_{k+1}&=
			\begin{cases}
				\theta_k\cdot \varsigma_\theta, & \rho_k<c_1,\\
				\min\{\theta_k/\varsigma_\theta,\theta_{\max}\}, & \text{otherwise},
			\end{cases}\\
			\tilde\eta_{k+1}&=
			\begin{cases}
				\eta_k\cdot \varsigma_\eta, & \rho_k<c_1,\\
				\min\{\eta_k/\varsigma_\eta,\eta_{\max}\}, & \text{otherwise},
			\end{cases}
		\end{align*}
		with prescribed $\varsigma_\gamma,\varsigma_\theta,\varsigma_\eta\in(0,1)$ and threshold $c_1>0$.
		The damping step above then ensures the required summability regardless of how
		$(\tilde\gamma_{k+1},\tilde\theta_{k+1},\tilde\eta_{k+1})$ are produced.
		
		Deterministic schedules are also covered. For example, fix $(\bar\gamma,\bar\theta,\bar\eta)$ in the interior of the intervals and set
		\begin{align*}
			\gamma_k
			&=\Pi_{[\gamma_{\min},\gamma_{\max}]}\big(\bar\gamma+\frac{a_\gamma}{(k+1)^{\delta_\gamma}}\big),\\
			\theta_k
			&=\Pi_{[\theta_{\min},\theta_{\max}]}\big(\bar\theta+\frac{a_\theta}{(k+1)^{\delta_\theta}}\big),\\
			\eta_k
			&=\Pi_{[\eta_{\min},\eta_{\max}]}\big(\bar\eta+\frac{a_\eta}{(k+1)^{\delta_\eta}}\big),
			\qquad
			\delta_\gamma,\delta_\theta,\delta_\eta>1,
		\end{align*}
		which yields convergent parameter sequences with summable deviations.
	\end{remark}
	
	Define $\calA^{\theta,\eta}, 	\calB^\theta,\calF_\gamma^{\theta,\eta}, \mathcal{F}_{\mu,\gamma}^{\theta,\eta}$ as in \eqref{eq:fcarp} and \eqref{eq:G-def}, then the main part of Algorithm~\ref{alg:nsfull} can be written compactly as the non-stationary Krasnosel'ski\u{\i}--Mann iteration
	\begin{equation}\label{eq:fp-nsfull}
		\z^{k+1}=\mathcal{F}_{\mu,\gamma_k}^{\theta_k,\eta_k}(\z^k).
	\end{equation}
	
	\begin{lemma}[Uniform averagedness and common fixed points]\label{lem:avg-fix-limit}
		Let $X,Y\subset\calH$ be nonempty closed convex with $X\cap Y\neq\emptyset$.
		Assume $\gamma_k\in(0,1)$, $\theta_k,\eta_k\in(0,1]$, and $\mu\in\bigl(0,\frac{2}{1+\gamma_k}\bigr)$.
		Then, $\mathcal{F}_{\mu,\gamma_k}^{\theta_k,\eta_k}$ is averaged, and
		\[
		\Fix\bigl(\mathcal{F}_{\mu,\gamma_k}^{\theta_k,\eta_k}\bigr)=X\cap Y.
		\]
		Moreover, if $\gamma_k\in[\gamma_{\min},\gamma_{\max}]\subset(0,1)$ and $\mu\in\bigl(0,\frac{2}{1+\gamma_{\max}}\bigr)$,
		then each $\mathcal{F}_{\mu,\gamma_k}^{\theta_k,\eta_k}$ is $\alpha_k$-averaged with
		$\alpha_k\le \frac{(1+\gamma_{\max})\mu}{2}<1$, i.e., the averagedness constant can be chosen uniformly over the parameter intervals.
	\end{lemma}
	
	\begin{proof}
		This is a direct specialization of Proposition~\ref{prop:fix-gen} and Item (i) of Theorem~\ref{th:avg-and-conv-gcarp}
		, whose proofs do not rely on stationarity of the iterates. 
	\end{proof}
	
	Following \cite{liang2016convergence}, we view the simultaneous non-stationarity in $(\gamma_k,\theta_k,\eta_k)$
	as a summable perturbation of a limiting averaged map. Under the increment conditions in Algorithm~\ref{alg:nsfull}, each of $\{\gamma_k\},\{\theta_k\},\{\eta_k\}$ is Cauchy and hence convergent.
	Let
	\[
	\bar\gamma:=\lim_{k\to\infty}\gamma_k,\qquad
	\bar\theta:=\lim_{k\to\infty}\theta_k,\qquad
	\bar\eta:=\lim_{k\to\infty}\eta_k.
	\]
	Then \eqref{eq:fp-nsfull} admits the perturbation decomposition
	\begin{equation}\label{eq:fp-nsfull-perturb}
		\z^{k+1}
		=\mathcal{F}_{\mu,\bar\gamma}^{\bar\theta,\bar\eta}(\z^k)
		+\underbrace{\Big(\mathcal{F}_{\mu,\gamma_k}^{\theta_k,\eta_k}-\mathcal{F}_{\mu,\bar\gamma}^{\bar\theta,\bar\eta}\Big)(\z^k)}_{=:~\ee^k}.
	\end{equation}

	The next lemma quantifies Lipschitz dependence on $(\theta,\eta)$ and will be used to bound $\|\ee^k\|$.
	
	\begin{lemma}[Parameter perturbation bounds]\label{lem:perturb-bounds}
		Fix $\z^\dagger\in X\cap Y$. For any $\z\in\calH$ and any $\theta,\theta'\in(0,1]$, $\eta,\eta'\in(0,1]$, we have
		\begin{align*}
			\|\mathcal R_X^\theta \z-\mathcal R_X^{\theta'}\z\|
			&\le 4|\theta-\theta'|\,\|\z-\z^\dagger\|,\\
			\|\calB^\theta \z-\calB^{\theta'}\z\|
			&\le 4|\theta-\theta'|\,\|\z-\z^\dagger\|,\\
			\|\calA^{\theta,\eta}\z-\calA^{\theta',\eta'}\z\|
			&\le 4\big(|\theta-\theta'|+|\eta-\eta'|\big)\,\|\z-\z^\dagger\|.
		\end{align*}
	\end{lemma}
	
	\begin{proof}
		Since $\z^\dagger\in X$, nonexpansiveness of $\proj_X$ yields
		\[
		\|\proj_X \z-\z\|
		\le \|\proj_X \z-\z^\dagger\|+\|\z-\z^\dagger\|
		\le 2\|\z-\z^\dagger\|.
		\]
		Moreover,
		\[
		\mathcal R_X^\theta \z-\mathcal R_X^{\theta'}\z
		=\bigl((1-\theta)-(1-\theta')\bigr)\z+\bigl(\theta-\theta'\bigr)(2\proj_X \z-\z)
		=2(\theta-\theta')(\proj_X \z-\z),
		\]
		hence $\|\mathcal R_X^\theta \z-\mathcal R_X^{\theta'}\z\|\le 4|\theta-\theta'|\,\|\z-\z^\dagger\|$.
        
		The bound for $\calB^\theta=\proj_Y\mathcal R_X^\theta$ follows from nonexpansiveness of $\proj_Y$.
		
		For $\calA^{\theta,\eta}=\tfrac12(\Id+\mathcal R_Y^\eta\mathcal R_X^\theta)$, write
		\[
		\mathcal R_Y^\eta\mathcal R_X^\theta \z-\mathcal R_Y^{\eta'}\mathcal R_X^{\theta'}\z
		=\big(\mathcal R_Y^\eta-\mathcal R_Y^{\eta'}\big)\mathcal R_X^\theta \z
		+\mathcal R_Y^{\eta'}\big(\mathcal R_X^\theta \z-\mathcal R_X^{\theta'}\z\big).
		\]
		Using nonexpansiveness of $\mathcal R_Y^{\eta'}$ and repeating the first part with $Y$ in place of $X$ since $\z^\dagger\in Y$, we obtain
		\[
		\big\|\big(\mathcal R_Y^\eta-\mathcal R_Y^{\eta'}\big)\w\big\|
		\le 4|\eta-\eta'|\,\|\w-\z^\dagger\|\qquad(\forall \w\in\calH),
		\]
		and since $\mathcal R_X^\theta$ is nonexpansive, $\|\mathcal R_X^\theta \z-\z^\dagger\|\le \|\z-\z^\dagger\|$.
		Combining these estimates and multiplying by $\tfrac12$ yields the desired bound for $\calA^{\theta,\eta}$. 
	\end{proof}
	
	Combining the perturbation decomposition \eqref{eq:fp-nsfull-perturb}, the summable error estimate implied by Lemma~\ref{lem:perturb-bounds}, and the standard Krasnosel'ski\u{\i}--Mann scheme with summable errors, we obtain the convergence of the fully non-stationary scheme.
	
	\begin{proposition}[Convergence of the fully non-stationary scheme]\label{prop:conv-nsfull}
		Assume $X\cap Y\neq\emptyset$ and $\calH=\bbR^n$.
		Let $\{\gamma_k\}\subset[\gamma_{\min},\gamma_{\max}]\subset(0,1)$,
		$\{\theta_k\}\subset[\theta_{\min},\theta_{\max}]\subset(0,1]$,
		$\{\eta_k\}\subset[\eta_{\min},\eta_{\max}]\subset(0,1]$,
		and $\mu\in\Bigl(0,\frac{2}{1+\gamma_{\max}}\Bigr)$.
		Assume further that there exist $(\bar\gamma,\bar\theta,\bar\eta)\in(0,1)\times(0,1]\times(0,1]$ such that
		\begin{equation}\label{eq:sum-params}
			\sum_{k=0}^\infty\Big(|\gamma_k-\bar\gamma|+|\theta_k-\bar\theta|+|\eta_k-\bar\eta|\Big)<\infty.
		\end{equation}
		Then Algorithm~\ref{alg:nsfull} generates a sequence $\{\z^k\}$ converging to a point $\z^\star\in X\cap Y$.
		Moreover, the shadow sequences $\x^{k+1}=\proj_X(\z^k)$ and $\y^{k+1}=\proj_Y(\mathcal R_X^{\theta_k}\z^k)$ satisfy
		\[
		(\z^k,\x^{k},\y^{k})\to (\z^\star,\z^\star,\z^\star).
		\]
	\end{proposition}
	
	\begin{proof}
		Fix any $\z^\dagger\in X\cap Y$ and define
		\[
		\mathcal F_k:=\mathcal{F}_{\mu,\gamma_k}^{\theta_k,\eta_k},
		\qquad
		\bar{\mathcal F}:=\mathcal{F}_{\mu,\bar\gamma}^{\bar\theta,\bar\eta}.
		\]
		
		By Lemma~\ref{lem:avg-fix-limit}, each $\mathcal F_k$ is averaged and hence nonexpansive, and
		$\Fix(\mathcal F_k)=X\cap Y$. In particular, $\mathcal F_k(\z^\dagger)=\z^\dagger$ for all $k$.
		Therefore,
		\[
		\|\z^{k+1}-\z^\dagger\|
		=\|\mathcal F_k(\z^k)-\mathcal F_k(\z^\dagger)\|
		\le \|\z^k-\z^\dagger\|
		\le \|\z^0-\z^\dagger\|,
		\]
		so $\{\z^k\}$ is bounded, and indeed Fej\'er monotone with respect to $X\cap Y$.
		
		From \eqref{eq:fp-nsfull-perturb} we have
		\[
		\z^{k+1}=\bar{\mathcal F}(\z^k)+\ee^k,
		\qquad
		\ee^k:=(\mathcal F_k-\bar{\mathcal F})(\z^k).
		\]
		Using $\mathcal{F}_{\mu,\gamma}^{\theta,\eta}=(1-\mu)\Id+\mu\big((1-\gamma)\calA^{\theta,\eta}+\gamma \calB^\theta\big)$,
		we expand
		\[
		\mathcal F_k-\bar{\mathcal F}
		=\mu\Big[(1-\gamma_k)\big(\calA^{\theta_k,\eta_k}-\calA^{\bar\theta,\bar\eta}\big)
		+\gamma_k\big(\calB^{\theta_k}-\calB^{\bar\theta}\big)
		+(\gamma_k-\bar\gamma)\big(\calB^{\bar\theta}-\calA^{\bar\theta,\bar\eta}\big)\Big].
		\]
		Since $\z^\dagger$ is a common fixed point of $\calA^{\bar\theta,\bar\eta}$ and $\calB^{\bar\theta}$ and both maps are
		nonexpansive, we obtain
		\[
		\|(\calB^{\bar\theta}-\calA^{\bar\theta,\bar\eta})(\z^k)\|
		=\|(\calB^{\bar\theta}-\calA^{\bar\theta,\bar\eta})(\z^k-\z^\dagger)\|
		\le 2\|\z^k-\z^\dagger\|.
		\]
		By Lemma~\ref{lem:perturb-bounds},
        \begin{align*}
            \|\calB^{\theta_k}\z^k-\calB^{\bar\theta}\z^k\|
		&\le 4|\theta_k-\bar\theta|\,\|\z^k-\z^\dagger\|,\\
		\|\calA^{\theta_k,\eta_k}\z^k-\calA^{\bar\theta,\bar\eta}\z^k\|
		&\le 4\big(|\theta_k-\bar\theta|+|\eta_k-\bar\eta|\big)\,\|\z^k-\z^\dagger\|.
        \end{align*}
		Combining these bounds and using $\gamma_k\in(0,1)$ gives
		\begin{align*}
		    \|\ee^k\|
		&\le
		C\,\mu\,\|\z^k-\z^\dagger\|\Big(|\gamma_k-\bar\gamma|+|\theta_k-\bar\theta|+|\eta_k-\bar\eta|\Big)\\
		&\le
		C\,\mu\,\|\z^0-\z^\dagger\|\Big(|\gamma_k-\bar\gamma|+|\theta_k-\bar\theta|+|\eta_k-\bar\eta|\Big),
		\end{align*}
		for an absolute constant $C>0$ (e.g., $C=8$).
		By \eqref{eq:sum-params}, $\sum_{k=0}^\infty \|\ee^k\|<\infty$.
	
		The iteration $\z^{k+1}=\bar{\mathcal F}(\z^k)+\ee^k$ is a Krasnosel'ski\u{\i}--Mann scheme for the averaged map
		$\bar{\mathcal F}$ with summable errors. Hence $\{\z^k\}$ is quasi-Fej\'er monotone with respect to
		$\Fix(\bar{\mathcal F})$ and converges in $\bbR^n$ to some $\z^\star\in \Fix(\bar{\mathcal F})$
		(see, e.g., \cite[Proposition~5.34]{bauschke2020correction}).
		By Lemma~\ref{lem:avg-fix-limit}, $\Fix(\bar{\mathcal F})=X\cap Y$, so $\z^\star\in X\cap Y$.

		Since $\proj_X$ is continuous, $\x^{k+1}=\proj_X(\z^k)\to \proj_X(\z^\star)=\z^\star$.
		Moreover, $\theta_k\to\bar\theta$ follows from \eqref{eq:sum-params}, and for any fixed $\z\in\bbR^n$,
		$\mathcal R_X^{\theta_k}\z\to \mathcal R_X^{\bar\theta}\z$.
		Thus,
		\[
		\y^{k+1}=\proj_Y(\mathcal R_X^{\theta_k}\z^k)\to \proj_Y(\mathcal R_X^{\bar\theta}\z^\star).
		\]
		Since $\z^\star\in X$, we have $\proj_X(\z^\star)=\z^\star$ and hence $\mathcal R_X^{\bar\theta}\z^\star=\z^\star$.
		Because $\z^\star\in Y$ as well, $\proj_Y(\z^\star)=\z^\star$, so the right-hand side equals $\z^\star$.
		Therefore, $\y^{k+1}\to \z^\star$.
	\end{proof}
	
	\begin{remark}[A convenient sufficient condition]\label{rmk:conv-nsfull-suff}
		A simple sufficient condition for \eqref{eq:sum-params} is the increment bound used in Algorithm~\ref{alg:nsfull}:
		if for some $\delta>0$,
		\[
		|\gamma_{k+1}-\gamma_k|\le \frac{c_\gamma}{(k+1)^{2+\delta}},\,\,
		|\theta_{k+1}-\theta_k|\le \frac{c_\theta}{(k+1)^{2+\delta}},\,\,
		|\eta_{k+1}-\eta_k|\le \frac{c_\eta}{(k+1)^{2+\delta}},
		\]
		then each parameter sequence is Cauchy and admits a limit $(\bar\gamma,\bar\theta,\bar\eta)$, and moreover
		$\sum_{k\ge0}|\gamma_k-\bar\gamma|<\infty$, $\sum_{k\ge0}|\theta_k-\bar\theta|<\infty$,
		$\sum_{k\ge0}|\eta_k-\bar\eta|<\infty$, implying \eqref{eq:sum-params}.
	\end{remark}

\section{Numerical experiments}\label{sec:numerics}
	
	We evaluate the proposed generalized composed alternating relaxed projection algorithm (\gcarp) and its non-stationary variant (ns-\gcarp) on several representative feasibility problems of the form
	$\find\,\x\in X\cap Y$ with nonempty intersection.
	The goal of this section is threefold:
	(i) to verify the theoretical rate characterization on subspace feasibility problems;
	(ii) to illustrate, both quantitatively and geometrically, how the additional relaxation parameters
	$(\theta,\eta)$ can further mitigate the spiraling behavior observed for Douglas--Rachford-type iterations;
	and (iii) to assess the practical benefit and robustness of full non-stationarity in challenging
	nonlinear feasibility settings.
	Throughout, we include as baselines the classical Douglas--Rachford method (DR) and its non-stationary
	variant (ns-DR) \cite{lorenz2019non}, the method of alternating projections (MAP), and several modern
	relaxed/averaged projection schemes (e.g., GRAP/RAAR/AAMR) whenever applicable.
	We also include the stationary and non-stationary \carp schemes from \cite{shen2025composed} to highlight
	the additional gains enabled by the proposed generalization.
	
	\subsection{General setup}\label{subsec:general-setup}
	
	\paragraph{Environment.}
	All experiments are performed in \textsc{Matlab} R2023a on a desktop computer with an Intel Core Ultra 9 185H processor (2.30 GHz) and 32 GB of RAM.
	The same random seed is used for all methods in each trial to ensure reproducibility.
	Unless stated otherwise, all vectors are initialized by i.i.d.\ standard normal entries and then normalized.
	
	\paragraph{Stopping criteria.}
	For a tolerance $\varepsilon>0$, we stop when both a fixed-point residual and a feasibility residual are small:
	\begin{equation}\label{eq:stop-criteria}
		\mathrm{FPR}^k:=\frac{\|\z^{k+1}-\z^k\|}{\max\{1,\|\z^k\|\}}\le \varepsilon,
	\end{equation}
	We also impose a maximum iteration cap $k_{\max}$ and report ``fail'' if the stopping criterion is not
	met within $k_{\max}$ iterations.

	\paragraph{Parameter selection.}
	\begin{itemize}
	\item \textbf{Relaxation parameter $\mu$.}
		For methods using an outer relaxation, we ensure $\mu\in\bigl(0,\frac{2}{1+\gamma_{\max}}\bigr)$ when $\gamma$ varies, to satisfy the averagedness requirement. In the stationary runs we typically use
		$\mu=1$ unless stated otherwise.
		\item \textbf{Stationary \gcarp parameters.}
		We consider $(\theta,\eta)\in(0,1]$ and $\gamma\in(0,1)$.
		For subspace problems, we test two tuning strategies:
		(i) a theory-driven choice based on the minimax guideline in
		Theorem \ref{th:gcarp-minimax-gamma} of symmetric regime, and
		(ii) a grid-search choice on a coarse grid
		\[
		\theta\in\{0.5,0.6,\ldots,1.0\},\,\,
		\eta\in\{0.5,0.6,\ldots,1.0\},\,\,
		\gamma\in\{0.0,0.05,\ldots,0.95\},
		\]
		where the best parameters are selected by minimizing the observed iteration count on a small validation
		instance and then fixed for the reported tests.
        
	\item \textbf{Fully non-stationary ns-\gcarp.}
		In Algorithm~\ref{alg:nsfull} we set the parameter intervals
		\[
		\gamma_k\in[\gamma_{\min},\gamma_{\max}],\quad
		\theta_k\in[\theta_{\min},\theta_{\max}],\quad
		\eta_k\in[\eta_{\min},\eta_{\max}],
		\]
		and enforce summable increments as described in Remark~\ref{rmk:nsfull-updates}.
		Unless otherwise specified, we use
		\[
		\gamma_{\min}=0.05,\ \gamma_{\max}=0.95,\,\,
		\theta_{\min}=\eta_{\min}=0.5,\ \theta_{\max}=\eta_{\max}=1,
		\,\,
		\delta=0.5.
		\]
		The trial parameters $(\tilde\gamma_{k+1},\tilde\theta_{k+1},\tilde\eta_{k+1})$ are produced by the
		multiplicative rule in Remark~\ref{rmk:nsfull-updates} with
		\[
		\varsigma_\gamma=\varsigma_\theta=\varsigma_\eta=0.95,\qquad c_1=0.9,
		\]
		and then damped using $\alpha_k=(k+1)^{-(2+\delta)}$.
	\end{itemize}
	
	\subsection{Feasibility of two subspaces}\label{subsec:two-subspaces}
	
	When $X$ and $Y$ are linear subspaces, all compared schemes are linear iterations and their linear
	convergence factors can be predicted from spectral information, as detailed in Section~\ref{sec:gcarp-opt}. In the following, we use the canonical construction with
	\[
	n=100,\qquad p=q=50,\qquad p+q=n,
	\]
	so that only the $2\times2$ principal-angle blocks appear in the spectral decomposition.
	Let $\phi_1\le \cdots\le \phi_p\in(0,\pi/2]$ be the desired principal angles, where $\phi_i$ is used to avoid confusion with the algorithmic parameter $\theta\in(0,1]$.
	Define the subspaces
	\begin{equation}\label{eq:subspaces-canonical}
		X:=\mathrm{span}\{\ee_1,\ldots,\ee_p\},
		\qquad
		Y:=\mathrm{span}\{c_i \ee_i+s_i \ee_{p+i}\}_{i=1}^p,
	\end{equation}
	where $c_i=\cos\phi_i$, $s_i=\sin\phi_i$ and $\ee_i$ is the $i$-th unit vector.
    
	Then the principal angles between $X$ and $Y$ are exactly $\{\phi_i\}_{i=1}^p$, and the orthogonal projectors are
	\begin{equation}\label{eq:PX-PY-canonical}
		\bmP_X=
		\begin{bmatrix}
			\bmI_p & \bm0\\
			\bm0 & \bm0
		\end{bmatrix},
		\
		\bmP_Y=
		\begin{bmatrix}
			\bmC^2 & \bmC\bmS\\
			\bmC\bmS & \bmS^2
		\end{bmatrix},
		\
		\bmC=\mathrm{diag}(c_i)_{i=1}^p,\ \bmS=\mathrm{diag}(s_i)_{i=1}^p .
	\end{equation}
	In all subspace experiments, we further impose
	\begin{equation}\label{eq:angle-schedule}
		\phi_p=\frac{\pi}{2},
		\qquad
		\phi_1=\phi_F,
		\qquad
		\phi_i=\phi_F+\frac{i-1}{p-1}\Bigl(\frac{\pi}{2}-\phi_F\Bigr),\quad i=1,\ldots,p.
	\end{equation}
	
	We structure the
	experiments into three parts:
	\begin{enumerate}
		\item \textbf{Predicted rates:} compare the predicted linear convergence factors as functions of the Friedrichs angle.
		\item \textbf{Numerical illustration:} confirm that empirical slopes match the predicted $\xi(\cdot)$ and compare
		practical iteration counts.
		\item \textbf{Trajectory illustration:} visualize how the additional relaxation parameters $(\theta,\eta)$ change
		the spiral/rotation behavior on principal-angle planes.
	\end{enumerate}
	Throughout this subsection, we fix $\mu=1$ to focus on the intrinsic contraction behavior of the fixed-point maps. For DR on subspaces, relaxation does not improve the optimal linear rate; see \cite{bauschke2016optimal}.

	\subsubsection{Predicted rates}\label{subsubsec:subspaces-predicted}

    \begin{figure}[t]
    \centering
    
    \begin{minipage}{0.465\textwidth}
        \centering
        \includegraphics[width=\linewidth]{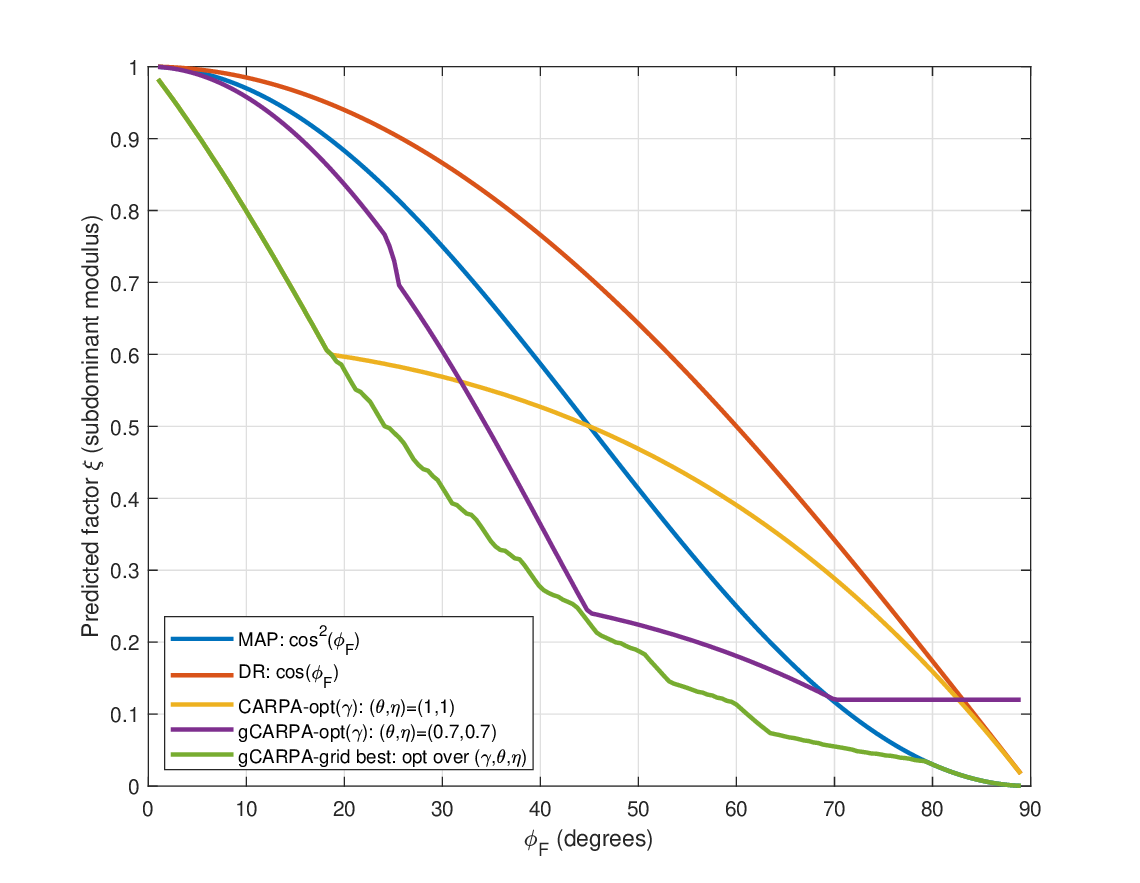}
        \caption{Predicted linear convergence factors on subspace feasibility problems with
			$n=100$, $p=q=50$, $\phi_p=\pi/2$, and angles $\{\phi_i\}$ defined by \eqref{eq:angle-schedule}.}
		\label{fig:ratecmp-gcarp}
    \end{minipage}
    \hspace{1em}
    \begin{minipage}{0.48\textwidth}
        \vspace{-0.9cm}
        \centering
        \includegraphics[width=1.05\linewidth]{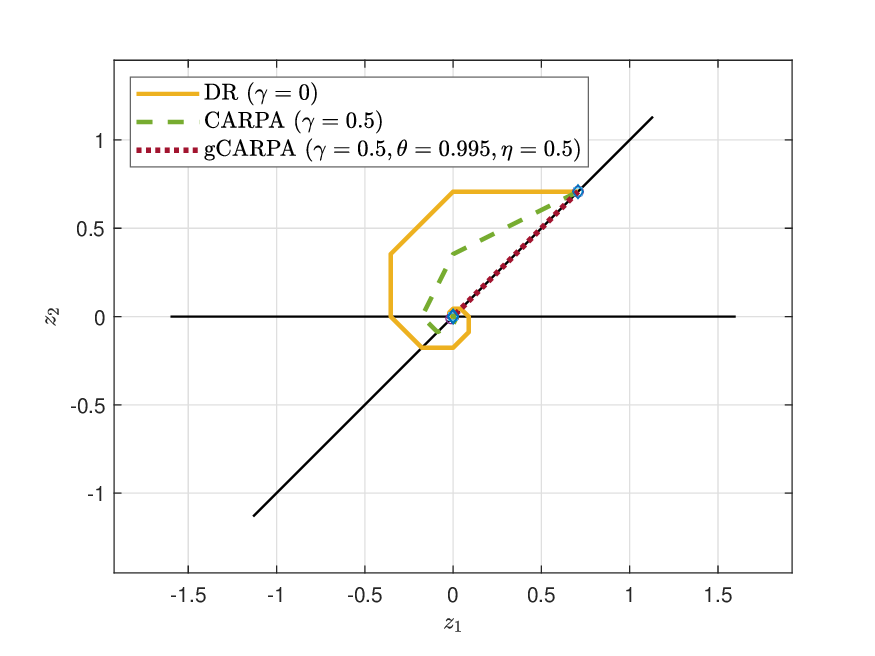}
        \vspace{-0.5cm}
        \caption{Trajectory on a single principal-angle plane with $\phi_F=45^\circ$.}
		\label{fig:subspaces-traj}
    \end{minipage}
    \end{figure}   
    
	We first compare predicted linear convergence factors on the subspace model
	\eqref{eq:subspaces-canonical}--\eqref{eq:angle-schedule}.
	For each method, the predicted factor is taken to be the subdominant-eigenvalue modulus
	$\xi(\cdot)$ defined in Section~\ref{sec:maths_bg}, i.e., the largest modulus of eigenvalues excluding~$1$.
	In our setting $X\cap Y=\{0\}$, hence the eigenvalue~$1$ does not occur and $\xi(\cdot)$ is simply the
	spectral radius.
	\paragraph{Predicted rates for MAP and DR.}
	For completeness we include the classical baselines:
	\begin{itemize}
		\item \textbf{MAP:} $\z^{k+1}=\proj_Y\proj_X\z^k$ has predicted factor $\xi_{\mathrm{MAP}}=\max_i \cos^2(\phi_i)$,
		so under \eqref{eq:angle-schedule} it equals $\cos^2(\phi_F)$.
		\item \textbf{DR:} $\z^{k+1}=\tfrac12(\Id+\calR_Y\calR_X)\z^k$ has predicted factor $\xi_{\mathrm{DR}}=\max_i \cos(\phi_i)$,
		so under \eqref{eq:angle-schedule} it equals $\cos(\phi_F)$; see \cite{bauschke2014rate,bauschke2016optimal}.
	\end{itemize}
	
	\paragraph{Predicted best over $\gamma$ rates for \carp and \gcarp.}
	We compare three envelopes with $\mu=1$ and the principal-angle list \eqref{eq:angle-schedule}:
	\begin{enumerate}
		\item \textbf{\carp-opt($\gamma$):} stationary \carp corresponds to $\theta=\eta=1$ in \gcarp.
		For each $\phi_F$ we compute the minimax-optimal $\gamma$ over $[0,1)$ using the
		critical-damping candidates from the \carp theory \cite{shen2025composed}, and set the predicted factor to
		\[
		\xi_{\mathrm{\carp}}^{\star}(\phi_F):=
		\min_{\gamma\in[0,1)}\ \max_{1\le i\le p}\rho\bigl(\bmM_i^{\mathrm{\carp}}(\gamma)\bigr).
		\]
		\item \textbf{\gcarp-opt($\gamma$) at fixed $(\theta,\eta)$:} we fix a representative pair
		\begin{equation}\label{eq:theta-eta-representative}
			(\theta,\eta)=(\theta_0,\eta_0):=(0.7,0.7),
		\end{equation}
		and compute the minimax-optimal $\gamma$ using the endpoint critical damping recipe
		in Theorem \ref{th:gcarp-minimax-gamma}.
		Denote the resulting envelope by $\xi_{\mathrm{\gcarp}}^{\star}(\phi_F;\theta_0,\eta_0)$.
		\item \textbf{\gcarp-grid best:} to visualize the potential gain of the generalization, we also report
		a coarse-grid lower envelope
		\[
		\xi_{\mathrm{\gcarp}}^{\mathrm{grid}}(\phi_F):=
		\min_{(\theta,\eta,\gamma)\in\Theta\times H\times \Gamma}\ \max_{1\le i\le p}\rho\bigl(\bmM_i(\gamma,\theta,\eta)\bigr),
		\]
		where $\Theta=H=\{0.5,0.6,0.7,0.8,0.9,1.0\},\,
		\Gamma=\{0,0.05,0.10,\ldots,0.95\}.$
		This curve is not claimed to be globally optimal; rather, it provides a practical upper bound on the optimal rate.
	\end{enumerate}

	We sweep $\phi_F$ on a fine grid in $(0,\pi/2)$, build the angle list \eqref{eq:angle-schedule}, and compute the
	predicted factors above. Figure~\ref{fig:ratecmp-gcarp} plots the predicted subdominant modulus $\xi$ as a function of the Friedrichs angle $\phi_F$.
	Besides the fixed-parameter curve gCARPA-opt$(\gamma)$ at $(\theta,\eta)=(0.7,0.7)$, we also report a grid-based
	gCARPA-grid best curve obtained by approximately optimizing over $(\gamma,\theta,\eta)$ for each $\phi_F$.
	The latter provides an empirical lower envelope of attainable linear factors within our parameter family and, on a broad
	range of intermediate angles, yields a visibly smaller $\xi$ than CARPA-opt$(\gamma)$, indicating that the additional
	relaxation degrees of freedom $(\theta,\eta)$ can further reduce the worst-case contraction predicted by the spectral model.
	
	To make the comparison explicit at representative angles, we also tabulate predicted factors at
	four values of $\phi_F$:
	\[
	\phi_F\in\Bigl\{\frac{\pi}{36},\frac{\pi}{18},\frac{\pi}{12},\frac{\pi}{9}\Bigr\},
	\quad\text{i.e., }5^\circ,10^\circ,15^\circ,20^\circ.
	\]
	Table~\ref{tab:subspaces-predicted} reports the predicted subdominant modulus.
	As expected, MAP and DR have factors close to $1$ for small $\phi_F$, reflecting slow convergence when the two subspaces are nearly parallel.
	In contrast, \carp with the minimax choice of $\gamma$ substantially reduces $\xi$, and the improvement becomes more pronounced as $\phi_F$ increases. For \gcarp, the behavior depends strongly on the choice of $(\theta,\eta)$.
	With the fixed setting $(\theta,\eta)=(0.7,0.7)$, the predicted factor remains close to the MAP/DR baselines, indicating that this choice does not provide effective damping of the rotation component for these angles.
	However, allowing $(\gamma,\theta,\eta)$ to vary and selecting the best combination on a grid yields factors comparable to \carp and even slightly better at $\phi_F=20^\circ$.
	This suggests that the additional degrees of freedom $(\theta,\eta)$ can further accelerate convergence, but only when they are tuned to the underlying principal-angle geometry. 
	
	\begin{table}[!htb]
		\centering
		\caption{Predicted linear factors for the subspace model
			\eqref{eq:subspaces-canonical}--\eqref{eq:angle-schedule} with $n=100$, $p=q=50$ and $\mu=1$.}
		\label{tab:subspaces-predicted}
		\begin{tabular}{c|cccc}
			\hline
			$\phi_F$ & $5^\circ$ & $10^\circ$ & $15^\circ$ & $20^\circ$\\
			\hline
			MAP  & 0.9924 & 0.9698 & 0.9330 & 0.8830 \\
			DR     &0.9962 & 0.9848 & 0.9659 & 0.9397\\
			\carp-opt$(\gamma)$   & 0.9056 & 0.7988 & 0.6830 & 0.5967 \\
			\gcarp-opt$(\gamma)$ &0.9894 & 0.9578 & 0.9062 & 0.8362\\
			\gcarp-grid best& 0.9064 & 0.7999 & 0.6840 & 0.5788\\
			\hline
		\end{tabular}
	\end{table}
    
	\subsubsection{Numerical illustration}\label{subsubsec:subspaces-numerical}
	We next verify the predicted rates by running the methods on the same subspace instances and plotting
	the decay of the fixed-point residual
	\[
	\mathrm{FPR}^k=\frac{\|\z^{k+1}-\z^k\|}{\max\{1,\|\z^k\|\}}.
	\]
	We use the canonical subspaces \eqref{eq:subspaces-canonical} with angle schedule \eqref{eq:angle-schedule}. We report results at two small and two moderate-to-large Friedrichs angles,
	\[
	\phi_F\in\Bigl\{\frac{\pi}{18},\frac{\pi}{9},\frac{\pi}{4},\frac{\pi}{3}\Bigr\},
	\quad\text{i.e., }10^\circ,20^\circ,45^\circ,60^\circ.
	\]
	The initialization is $\z^0\sim\mathcal N(\bm0,\bmI_n)$ normalized to $\|\z^0\|=1$.
	We stop when $\mathrm{FPR}^k\le 10^{-12}$ or when the maximum number of iterations $k_{\max}=5000$ is reached.

	Figure~\ref{fig:subspaces-curves} reports the decay of the fixed-point residual
	$\mathrm{FPR}^k$ for DR, \carp, and \gcarp. The main purpose of these plots is to validate that the asymptotic linear convergence is governed by the eigen-structure characterized in Section~\ref{subsubsec:subspaces-predicted}.
	For each run we estimate an empirical linear factor $\hat r$ by fitting a line to
	$\log(\mathrm{FPR}^k)$ over a tail window where the decay is approximately geometric, and compare $\hat r$ with the predicted subdominant-eigenvalue modulus $\xi(\A)$ computed from the principal-angle
	block eigenvalues.
	For instance, at $\phi_F=10^\circ$ the theory predicts $\xi_{\mathrm{DR}}=\cos(10^\circ)\approx 0.9848$ and the fitted value reported in the legend gives
	$\hat r\approx 0.985$ for DR. Similarly close agreement is observed for \carp and \gcarp, up to the usual
	transient phase at early iterations.
	
	\begin{figure}[t]
		\centering
		\begin{subfigure}[t]{0.48\linewidth}
			\centering
			\includegraphics[width=\linewidth]{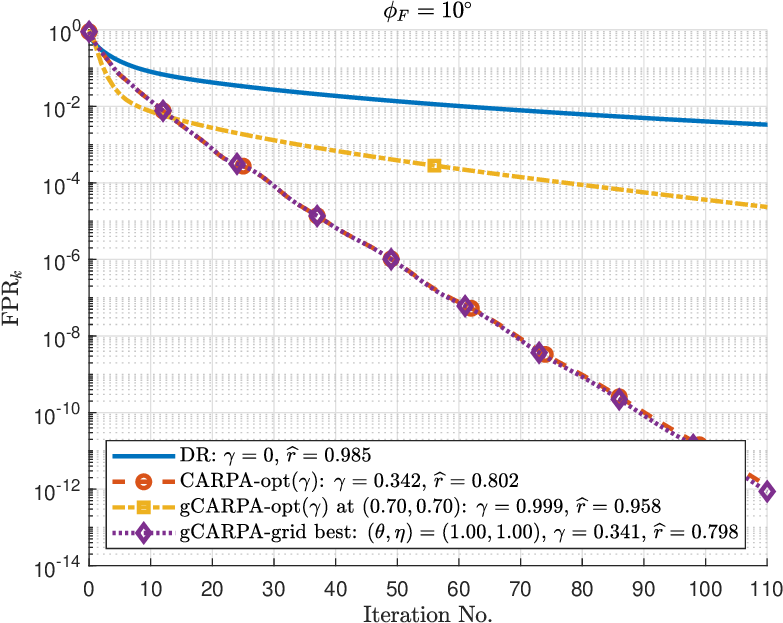}
			\caption{$\phi_F=10^\circ$.}
		\end{subfigure}\hfill
		\begin{subfigure}[t]{0.48\linewidth}
			\centering
			\includegraphics[width=\linewidth]{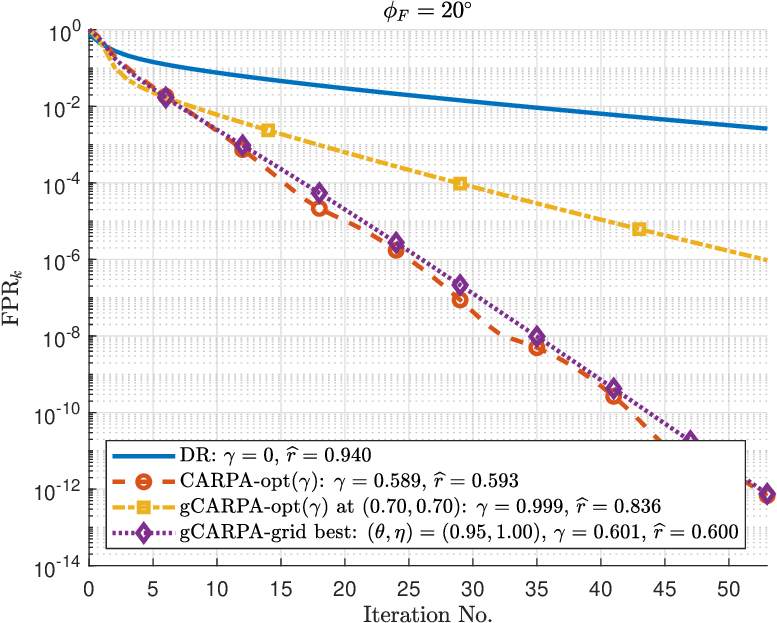}
			\caption{$\phi_F=20^\circ$.}
		\end{subfigure}
		
		\vspace{1mm}
		
		\begin{subfigure}[t]{0.48\linewidth}
			\centering
			\includegraphics[width=\linewidth]{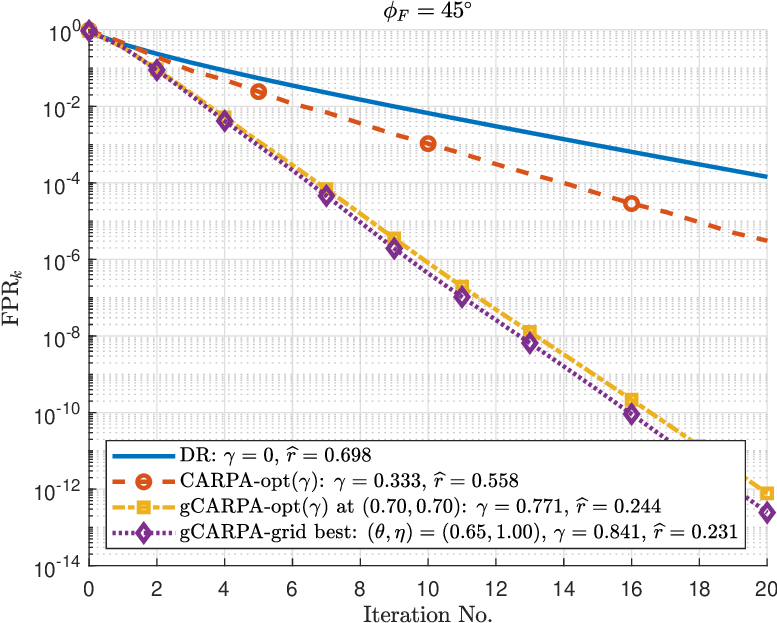}
			\caption{$\phi_F=45^\circ$.}
		\end{subfigure}\hfill
		\begin{subfigure}[t]{0.48\linewidth}
			\centering
			\includegraphics[width=\linewidth]{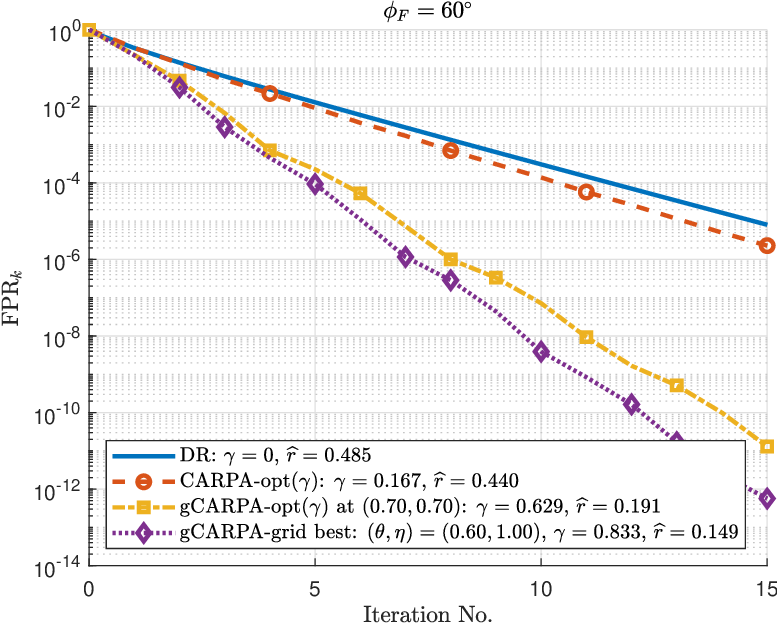}
			\caption{$\phi_F=60^\circ$.}
		\end{subfigure}
		\caption{Subspace feasibility with $n=100$, $p=q=50$, $\phi_p=\pi/2$, and angles given by \eqref{eq:angle-schedule}.
			Log-scale plots of $\mathrm{FPR}^k=\|\z^{k+1}-\z^k\|/\max\{1,\|\z^k\|\}$ for DR, \carp-opt$(\gamma)$ with $\theta=\eta=1$, \gcarp-opt$(\gamma)$ with $\theta=\eta=0.7$,
			and \gcarp-grid best $(\gamma,\theta,\eta)$. }
		\label{fig:subspaces-curves}
	\end{figure}
	
	Beyond theory validation, the plots in Figure~\ref{fig:subspaces-curves} provide a direct performance comparison. As expected, \carp consistently improves upon DR by substantially reducing the asymptotic factor,
	which is reflected by the visibly steeper tail slopes on the log-scale plots.
	For \gcarp, the benefit is angle- and parameter-dependent: a fixed choice such as $(\theta,\eta)=(0.7,0.7)$ is not
	uniformly competitive, which can be close to DR for small $\phi_F$. However, for larger angles, e.g., $\phi_F=45^\circ,60^\circ$, it yields remarkably smaller factors than \carp in our tests.
	Moreover, the grid-best tuning demonstrates that allowing $(\theta,\eta)$ to vary can match \carp in the small-angle regime
	and can outperform it in the large-angle regime, highlighting the additional flexibility and practical benefits of the generalized relaxation parameters.
	
	To complement the asymptotic-rate plots, we report in Table \ref{tab:subspaces-iter} the number of iterations
	required to reach the common stopping tolerance $\mathrm{FPR}^k\le 10^{-12}$.
	In particular, DR becomes dramatically slower when the Friedrichs angle is small, e.g., $10^\circ$,
	whereas \carpa reduces the iteration count by more than an order of magnitude.
	Moreover, the additional degrees of freedom in \gcarp translate into further savings in the large-angle regimes of $\phi_F\in\{45^\circ,60^\circ\}$, and \gcarp-grid best attains the target tolerance in substantially fewer
	iterations than \carpa, illustrating the benefit of tuning $(\theta,\eta)$ beyond the classical $\theta=\eta=1$ case.
	
	\begin{table}[h]
		\centering
		\caption{Iterations to reach $\mathrm{FPR}^k\le 10^{-12}$ on the subspace model
			\eqref{eq:subspaces-canonical}--\eqref{eq:angle-schedule} with $n=100$, $p=q=50$, $\mu=1$.}
		\label{tab:subspaces-iter}
		\begin{tabular}{c|cccc}
			\hline
			$\phi_F$ & $10^\circ$ & $20^\circ$ & $45^\circ$ & $60^\circ$\\
			\hline
			DR    & \texttt{1532} & \texttt{400} & \texttt{74} & \texttt{39}\\
			\carp-opt$(\gamma)$  & \texttt{112} & \texttt{54} & \texttt{48} & \texttt{35}\\
            \gcarp-opt($\gamma$) & \texttt{505} & \texttt{131} & \texttt{21} & \texttt{18}\\
			\gcarp-grid best & \texttt{111} & \texttt{54} & \texttt{21} & \texttt{16}\\
			\hline
		\end{tabular}
	\end{table}
	
	\subsubsection{Trajectory illustration}\label{subsubsec:subspaces-trajectory}
	
	Finally, we visualize the trajectory behavior on a single principal-angle plane.
	For subspace feasibility, the DR fixed-point iteration typically exhibits a pronounced
	spiral/rotation component; see, e.g., \cite{bauschke2014rate,poon2019trajectory}, which is caused by
	complex conjugate eigenvalues of the corresponding $2\times2$ principal-angle block.
	One motivation for \carp in \cite{shen2025composed} is to mitigate this rotational behavior by blending
	a DR-type mapping with a projection--reflection-type mapping.
	
	Our generalization introduces two additional relaxation parameters $(\theta,\eta)$, which provide
	an extra degree of freedom to control the balance between rotation and damping on each principal-angle plane.
	To highlight this effect clearly, in the plot below we fix the same mixing parameter $\gamma$ for both \carp
	and \gcarp, and only change $(\theta,\eta)$. In other words, the visualization is designed to show that
	even under a fixed $\gamma$, adjusting $(\theta,\eta)$ can substantially reduce the rotation component
	by moving the block spectrum from a complex-eigenvalue regime to a damped real-eigenvalue regime.

	We consider two lines $X$ and $Y$ in $\mathbb{R}^2$ forming a principal angle $\phi_F=45^\circ$.
	We initialize
	\[
	\z^0=\frac{1}{\sqrt{2}}(1,1),
	\]
	and plot the trajectory $\{\z^k\}$ in the $(\z_1,\z_2)$ plane. Figure~\ref{fig:subspaces-traj} visualizes the fixed-point iterates on a single principal-angle plane.
	As predicted by the spectral analysis for subspaces, the DR mapping typically
	induces a noticeable rotation component. This rotation component
	is one of the reasons why DR can be inefficient on subspace feasibility instances.
	
	In contrast, the proposed schemes reduce the rotation and produce a more damped approach to the
	intersection. \carp already mitigates the rotational effect by blending a DR-type map with a
	projection--reflection-type step, while \gcarp further introduces the relaxation parameters $(\theta,\eta)$, providing additional degrees of freedom to adjust the balance between rotation and damping on the
	principal-angle plane. 
    
	\subsection{Feasibility of a ball tangent to a line}\label{subsec:ball-line}
	
	We next consider a simple yet notoriously difficult nonlinear feasibility instance in $\bbR^2$:
	the intersection of the unit Euclidean ball and an affine line tangent to it.
	This geometry is known to be challenging for projection-type schemes, as the two sets meet
	with zero angle at the unique intersection point, leading to a long slow tail phase.
	In particular, the practical winner is often problem- and tolerance-dependent, and non-stationary
	steps can be significantly more effective than any fixed parameter choice.
	
	Let	$	Y:=\{\x\in\bbR^2:\ \|\x\|\le 1\},
	\,\,
	X:=\{\x\in\bbR^2:\ \bma^\top \x = 1\},
	\,\,
	\bma:=\tfrac{1}{\sqrt2}(1,1)^\top.$
	Then $X$ is tangent to $Y$ and $X\cap Y=\{\x^\star\}$ with $\x^\star=\bma$.
	For each trial we draw $\uu\sim\mathcal N(\bm0,\bmI_2)$, normalize it to $\|\uu\|=1$ and set $	\z^0=\x^\star+10\uu.$
	We run each method until the fixed-point residual $\mathrm{FPR}^k$
	falls below a target tolerance $\mathrm{tol}\in\{10^{-4},10^{-6},10^{-8},10^{-10}\}$, or until
	$k_{\max}=10^4$ iterations are reached. Reported iteration counts are averaged over $10^4$ random starts.
	
	We compare DR, ns-DR from \cite{lorenz2019non}, MAP and GRAP.
	We also include \carp, ns-\carp and the generalized variants: \gcarp, ns-\gcarp.

	Table~\ref{tab:ball-line-iter} summarizes the averaged iteration counts.
	As expected for the tangency geometry, stationary projection/reflection schemes exhibit a pronounced slow tail:
	MAP and GRAP are orders of magnitude slower and frequently approach the iteration cap at tight tolerances,
	and stationary \carp also deteriorates rapidly as $\mathrm{tol}$ decreases.
	In contrast, allowing non-stationarity is crucial on this instance: ns-DR consistently achieves the smallest
	iteration counts across all tolerances, and ns-\carp provides a substantial improvement over its stationary
	counterpart, e.g., from $3028$ to $231$ at $\mathrm{tol}=10^{-6}$. \gcarp does not automatically bring gains here; with the present tuning it
	essentially matches DR across all tolerances, indicating that extra relaxation parameters $(\theta,\eta)$ must be adapted in a geometry-aware way to be beneficial.
	Finally, the fully non-stationary scheme ns-\gcarp shows its advantage mainly in the high-accuracy regime. While it is comparable to ns-\carp at coarse tolerances, it becomes faster when $\mathrm{tol}$ is stringent,
	reducing the average iteration count from $499$ to $411$ at $\mathrm{tol}=10^{-8}$ and from $727$ to $558$ at
	$\mathrm{tol}=10^{-10}$. This suggests that adapting $(\theta_k,\eta_k)$ in addition to $\gamma_k$ can further
	mitigate the tail slowdown near the tangency point.
	
	\begin{table}[!htbp]
		\centering
		\caption{Ball tangent to a line in $\bbR^2$.
			Comparison of number of steps required to reach stopping tolerance.
			( -- means the method hits $k_{\max}$ for at least one start and the average is not reported).}
		\label{tab:ball-line-iter}
		\begin{tabular}{c|cccccccc}
			\hline
			$\mathrm{tol}$ & DR & ns-DR & MAP & GRAP & \carp & ns-\carp & \gcarp & ns-\gcarp \\
			\hline
			$10^{-4}$  & 24   & 15   & 292  & 178  & 103  & 62   & 24   & 63  \\
			$10^{-6}$  & 169  & 21   & 6293 & 4499 & 3028 & 231  & 169  & 199 \\
			$10^{-8}$  & 711  & 28   & 9997 & 9947 & 9142 & 499  & 711  & 411 \\
			$10^{-10}$ & 1036 & 35   & --   & 9992 & 9810 & 727  & 1036 & 558 \\
			\hline
		\end{tabular}
	\end{table}
    
	\subsection{Sparse linear inverse problems}\label{sec:cs}
	We consider the compressed-sensing feasibility model
	\begin{equation}\label{eq:bp-gcarp}
		X=\{\x\in\mathbb{R}^n:\ \A\x=\bmb\},
		\qquad
		Y=\{\x\in\mathbb{R}^n:\ \|\x\|_1\le c\},
	\end{equation}
	where $\A\in\mathbb{R}^{m\times n}$ has full row rank. Given a $\kappa$-sparse vector
	$\x_{\mathrm{sol}}\in\mathbb{R}^n$, we set $\bmb=\A\x_{\mathrm{sol}}$ and $c=\|\x_{\mathrm{sol}}\|_1$.
	The projection onto the affine set $X$ is
	\[
	\proj_X(\w)=\w-\A^\top(\A\A^\top)^{-1}(\A\w-\bmb),
	\]
	and the projection onto the $\ell_1$-ball $Y$ can be computed efficiently  \cite{condat2016fast,duchi2008efficient}.
	
	In all methods tested below, the dominant cost per iteration is the application of $\A$ and $\A^\top$
	together with one $\ell_1$-ball projection. Since $\A\A^\top\in\mathbb{R}^{m\times m}$ is fixed and
	positive semi-definite, we precompute a Cholesky factorization of $\A\A^\top$ once, and apply $(\A\A^\top)^{-1}$ via two
	triangular solves when evaluating $\proj_X(\cdot)$.
	
	\paragraph{A toy example}\label{sec:cs-toy}
	We first present a toy instance to compare the methods and to visualize the effect of the additional
	relaxation parameters in \gcarp.
	We set $(m,n)=(500,2000)$ and $\kappa=50$. The entries of $\A$ are sampled from the standard normal
	distribution and scaled by $1/\sqrt{m}$. The sparse ground truth $\x_{\mathrm{sol}}$ is generated by
	selecting a random support of size $\kappa$ and assigning i.i.d.\ standard normal entries on that support;
	then $\bmb=\A\x_{\mathrm{sol}}$ and $c=\|\x_{\mathrm{sol}}\|_1$, so that $\x_{\mathrm{sol}}\in X\cap Y$.
	
	We initialize $\z^0\sim\mathcal{N}(\bm0,\bmI_n)$ and normalize it to $\|\z^0\|=1$.
	We report:
	(i) the fixed-point residual $\mathrm{FPR}^k=\frac{\|\z^{k+1}-\z^k\|}{\max\{1,\|\z^k\|\}},$ shown on a logarithmic scale; and
	(ii) the support size of the $\ell_1$-projection output used by each method, i.e.,
	$\mathrm{card}(\mathrm{supp}(\y^k))$ where $\y^k:=\proj_Y(\cdot)$ is the $Y$-projection produced in
	the corresponding iteration.
	Intuitively, once the iterates enter the correct face of the $\ell_1$-ball, the support size tends to stabilize
	near $\kappa$, while oscillations in $\mathrm{card}(\mathrm{supp}(\y^k))$ reflect active-set changes of the
	$\ell_1$-projection during the transient phase.

	Figure~\ref{fig:cs-toy} summarizes the results.
	The log-scale $\mathrm{FPR}^k$ curves in Figure~\ref{fig:cs-toy}(a) show a clear separation between MAP and the projection--reflection-type methods (GRAP/\carpa/\gcarp and their non-stationary variants), which reach
	the numerical accuracy floor within a few hundred iterations on this instance.
	In Figure~\ref{fig:cs-toy}(b), MAP exhibits a long plateau with a large support before the $\ell_1$-projection output
	eventually sparsifies, whereas the faster methods identify a sparse support much earlier.
	Comparing \carpa with \gcarp, the additional relaxation parameters $(\theta,\eta)$ provide extra freedom to reduce
	the transient oscillations in the $Y$-step support and to accelerate the drop of $\mathrm{FPR}^k$ on this instance.
	Similarly, the ns-gCARPA can be viewed as a robustness mechanism: it retains the
	projection-only structure while allowing the method to adaptively move toward a favorable parameter regime without
	requiring a hand-picked fixed triple.
	
	\begin{figure}[t]
		\centering
		\subfloat[Convergence of $\mathrm{FPR}^k$]{\includegraphics[width=0.45\textwidth]{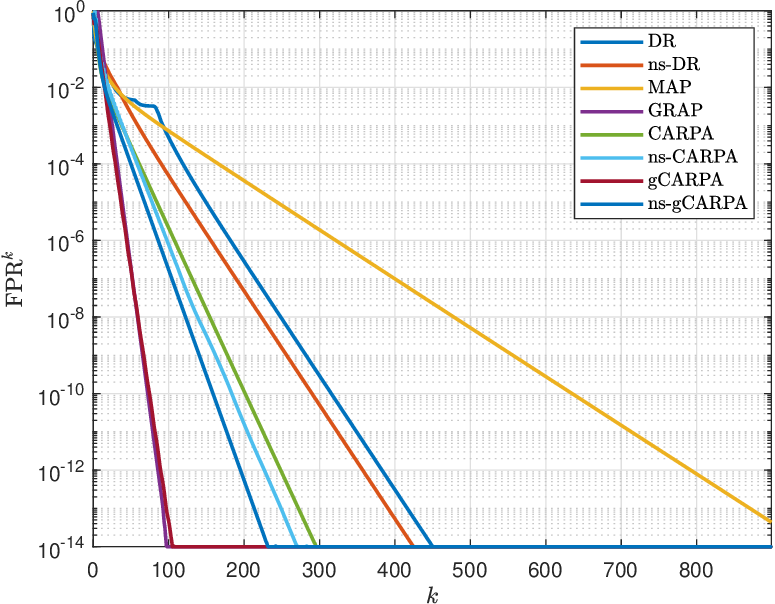}}
		\hskip3mm
		\subfloat[Support size of the $P_Y(\cdot)$ output]{\includegraphics[width=0.45\textwidth]{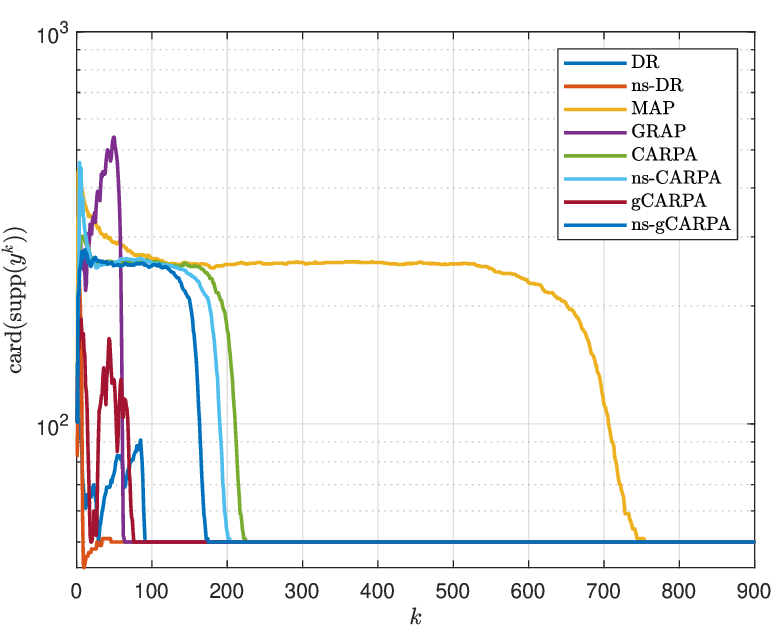}}
		\caption{Toy sparse linear inverse instance with $(m,n)=(500,2000)$ and $\kappa=50$.
			(a) Log-scale fixed-point residual $\mathrm{FPR}^k$.
			(b) Evolution of the support size $\mathrm{card}(\mathrm{supp}(\y^k))$ of the $\ell_1$-projection output
			$\y^k=P_Y(\cdot)$ used by each method.}
		\label{fig:cs-toy}
	\end{figure}

	\paragraph{Realistic examples}\label{sec:cs-real}
    We next consider more realistic sensing models commonly used in compressed sensing.
	Let $\bmf\in\mathbb{R}^n$ be a signal that admits a sparse representation under a dictionary $\bmB\in\mathbb{R}^{n\times n}$,
	i.e., $\bmf=\bmB\x_{\mathrm{sol}}$ for some sparse $\x_{\mathrm{sol}}$.
	Given a measurement operator $\bmM\in\mathbb{R}^{m\times n}$ and measurements $\bmb=\bmM\bmf=\bmM\bmB\x_{\mathrm{sol}}$,
	recovering $\bmf$ from $\bmb$ can be formulated as \eqref{eq:bp-gcarp} with $\bmA=\bmM\bmB$ and $c=\|\x_{\mathrm{sol}}\|_1$.
	We adopt four standard problem settings summarized in Table~\ref{tab:cs-p-gcarp}; see, e.g., \cite{van2009algorithm} for details.

    \begin{table}[b]
		\centering
        \caption{Problem settings for realistic sparse linear inverse experiments.}
		\begin{tabular}{|c|c|c|c|c|}
			\hline
			problem & $(m,n)$ & $\kappa$ & $\bmB$ & $\bmM$\\ \hline
			1 & (1024,2048) & 120 & DCT & Dirac\\ \hline
			2 & (600,2560)  & 20  & Id  & Gaussian\\ \hline
			3 & (256,1024)  & 32  & Id  & Gaussian\\ \hline
			4 & (200,1000)  & 3   & DCT & Restriction\\ \hline
		\end{tabular}
		\label{tab:cs-p-gcarp}
	\end{table}
	We compare the same set of methods under the same parameter choices as in the toy example.
	For each problem, we plot the decay of the fixed-point residual
	$\mathrm{FPR}^k$, as shown in Figure~\ref{fig:cs-real}.
	In all cases, the curves exhibit an eventually linear regime,
	which is consistent with the averaged-operator framework used in our convergence analysis.
	
	\begin{figure}[t]
		\centering
		\subfloat[Problem 1]{\includegraphics[width=0.465\textwidth]{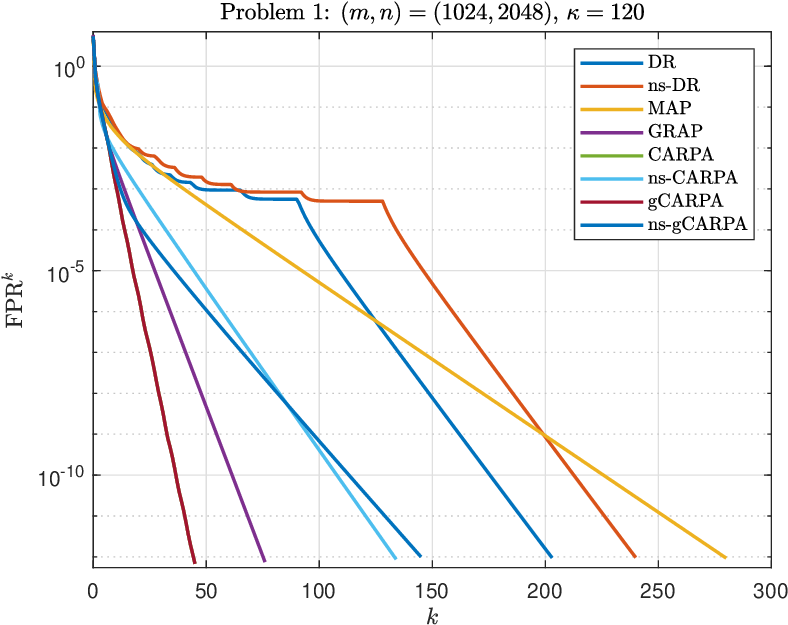}\label{fig:CS-1}}
		\hskip3mm
		\subfloat[Problem 2]{\includegraphics[width=0.465\textwidth]{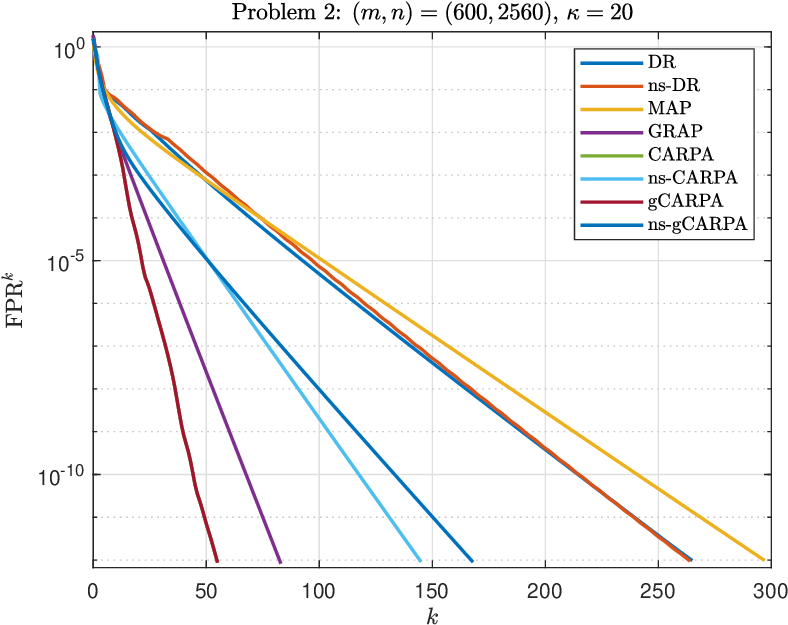}\label{fig:CS-2}}
		\\
		\subfloat[Problem 3]{\includegraphics[width=0.465\textwidth]{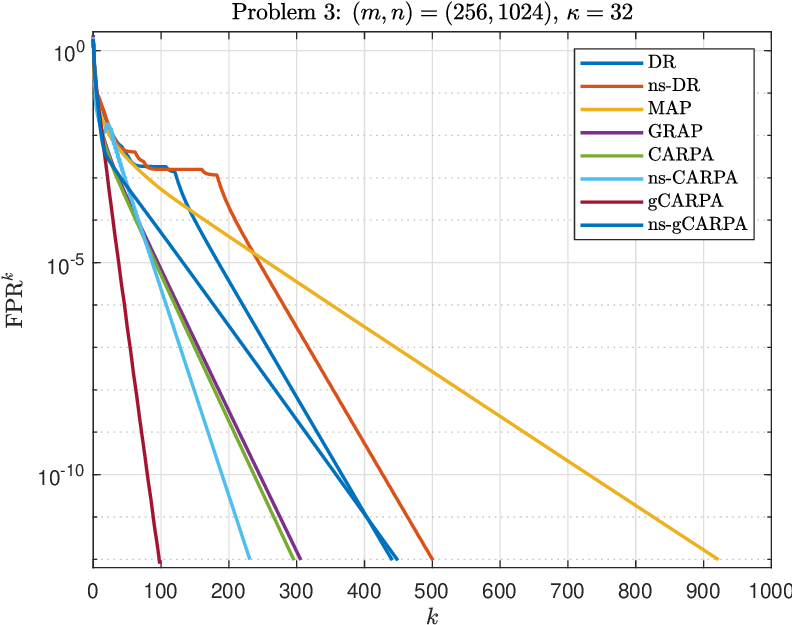}\label{fig:CS-3}}
        \hskip3mm
		\subfloat[Problem 4]{\includegraphics[width=0.465\textwidth]{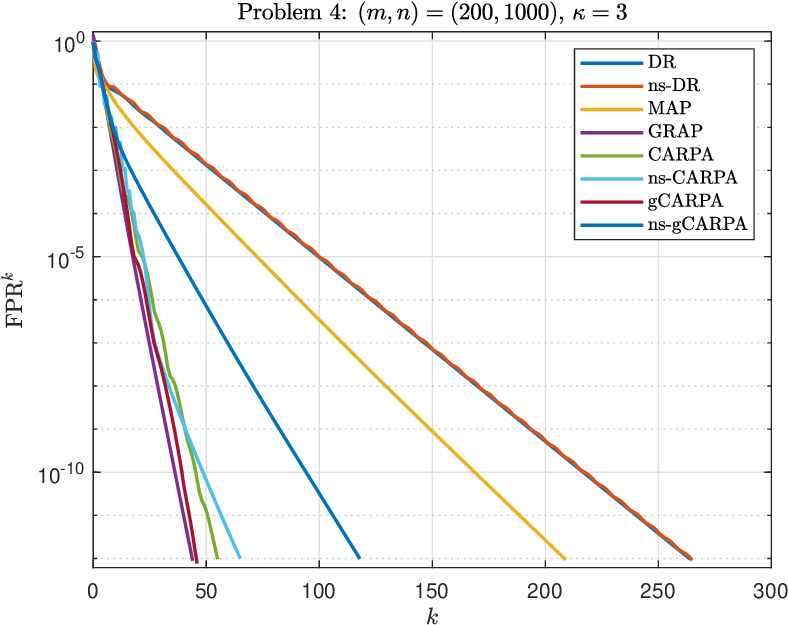}\label{fig:CS-4}}
		\caption{Realistic sparse linear inverse instances: log-scale $\mathrm{FPR}^k$ curves of different methods.}
		\label{fig:cs-real}
	\end{figure}
    
	Figure \ref{fig:cs-real} highlights two qualitative features that are typical for large-scale sparse linear inverse feasibility.
	First, the relative performance of the methods is strongly instance dependent.
	For example, in Problem 1, the DR/ns-DR trajectories display a noticeable transient behavior, characterized by a plateau-like behavior before the onset of final decay, whereas projection--reflection based schemes
	(GRAP, \carpa/ns-\carpa) enter the fast linear regime earlier.
	In contrast, in Problems~2--4 the curves are closer to clean linear decay throughout most iterations,
	and the gap among methods is primarily reflected by the slopes of the log-scale lines.
	
	Second, across all four problems, \carpa and especially ns-\carpa remain among the most competitive schemes,
	in agreement with the conclusions reported in \cite{shen2025composed} for these benchmark settings.
	Our generalized variants are at least competitive with the baseline methods on these instances:
	\gcarp improves over or matches \carpa in several regimes, e.g., Problems 1 and 3,
	and ns-\gcarp is consistently stable and competitive,
	often tracking the best-performing curve once the iterates have entered the asymptotic regime.
	This behavior supports the main algorithmic motivation of \gcarp:
	the extra relaxation parameters $(\theta,\eta)$ and their non-stationary adaptation provide additional degrees of freedom
	to better match the local geometry induced by the sensing model $(\bmB,\bmM)$, without sacrificing the simplicity of projection steps.

    \begin{table}[h]
		\centering
		\caption{Iterations to reach $\mathrm{FPR}^k\le 10^{-12}$ for the realistic sparse linear inverse instances in Table \ref{tab:cs-p-gcarp}.}
		\label{tab:cs-real-iter}
		\begin{tabular}{c|cccccccc}
			\hline
			problem & DR & ns-DR & MAP & GRAP & \carpa & ns-\carpa & \gcarp & ns-\gcarp\\
			\hline
			1 & 204 & 241 & 281 & 77  & 46  & 135 & 46  & 146 \\
			2 & 266 & 265 & 298 & 84  & 56  & 146 & 56  & 169 \\
			3 & 442 & 502 & 923 & 307 & 297 & 232 & 99  & 450 \\
			4 & 265 & 266 & 210 & 45  & 56  & 66  & 47  & 119 \\
			\hline
		\end{tabular}
	\end{table}

    Table \ref{tab:cs-real-iter} reports the number of iterations required to reach $\mathrm{FPR}^k \le 10^{-12}$ on four realistic sparse linear inverse instances.
    Overall, the performance is clearly instance dependent. On Problems 1--2, \carpa and \gcarp achieve the best results, requiring 46 and 56 iterations, respectively. On Problem 4, GRAP is the fastest with 45 iterations, followed closely by \gcarp with 47 iterations.
    The most significant improvement of the generalized scheme is observed on Problem 3, where \gcarp converges in 99 iterations, substantially outperforming \carpa and ns-\carpa, which require 297 and 232 iterations, respectively. This suggests that the additional relaxation degrees of freedom can be advantageous under certain sensing geometries.
    Finally, the fully non-stationary variant ns-\gcarp does not consistently outperform \gcarp in these experiments, indicating that aggressive parameter variation may introduce additional transient effects in $\ell_1$-ball feasibility problems.
	
	\section{Conclusion}\label{sec:conclusion}
	In this paper, we proposed the generalized composed alternating relaxed projection algorithm (\gcarp) for the two-set convex feasibility problem,
	together with its non-stationary variants ns-gCARPA.
	The algorithm unifies and extends several classical fixed-point iterations, and in particular recovers \carpa
	as a special case.
	On the theoretical side, we established their global convergence under standard assumptions.
	For the subspace model, we derived a principal-angle spectral characterization that yields explicit predicted linear factors
	and supports principled parameter selection, including a minimax/critical-damping recipe in a symmetric regime.
	Numerically, experiments on subspaces, the ball--line instance, and sparse linear inverse problems illustrate that the
	additional relaxation parameters $(\theta,\eta)$ can be beneficial in certain regimes, and that non-stationary tuning can improve robustness across instances.

\bibliographystyle{plain}
\bibliography{bib_Arxiv}
\end{document}